\documentclass[10pt,twocolumn,twoside]{IEEEtran}
\IEEEoverridecommandlockouts 
\usepackage{graphicx}    
\usepackage{cite }  
\usepackage{color}
\usepackage{amsmath}
\usepackage{amsthm}
\usepackage{amstext}
\usepackage{amssymb}
\usepackage{mathabx}
\usepackage{eufrak} 
\usepackage{amsfonts}     
\usepackage{dsfont}       
\newcommand{\halmos}	{${ }^{ }$\hfill\rule{2mm}{2mm}}

\newcommand{\EQ}{\begin{eqnarray}}
\newcommand{\EN}{\end{eqnarray}}
\newcommand{\EQQ}{\begin{eqnarray*}}
\newcommand{\ENN}{\end{eqnarray*}}
\newcommand{\nnum}{\nonumber}

\title{Synchronization on Lie Groups: Coordination of Blind Agents \thanks{ This work was supported by the Australian Research Council.}}

\author{Farzin Taringoo \thanks{Department of Electrical and Electronic Engineering, The University of Melbourne, ftaringoo@unimelb.edu.au.}}
\begin{document}
\maketitle
\begin{abstract}                
This  paper presents an  algorithm for the synchronization of blind agents (agents are unable to observe other agents, i.e. no communication) evolving on  a connected Lie group $G$. We employ the method of extremum seeking control for  nonlinear dynamical systems defined on connected Riemannian manifolds to achieve the synchronization among the agents. This approach is independent of the underlying graph of the system and each agent updates its position on $G$ by only receiving the synchronization cost function. The results are obtained by employing the notion of geodesic dithers for extremum seeking on Riemannian manifolds and their equivalent version on Lie groups and applying Taylor expansion of smooth functions on Riemannian manifolds. We apply the obtained results to  synchronization problems defined on Lie groups $SO(3)$ and $SE(3)$ to demonstrate the efficacy of the proposed algorithm.
\end{abstract}

 
\section{Introduction}
Synchronization is an important topic in analysis of multi-agent systems, see \cite{rah, sar,sar1, sar2, sca,sep, cor, jad,olf, are, law}. This problem may arise as the behavior of agents in nature. The synchronization problem has been extensively analyzed from control and optimization point of view, see \cite{jad,olf}. Various aspects such as optimality of configurations, collision avoidance and mean field stochastic games have been studied for this class of problems. 
Synchronization of agents  is closely related to the \textit{consensus} problem in which agents minimize the summation of their local objective functions, see \cite{jad,olf}. Depending on cost functions defined for the network of agents, a synchronization problem can be converted to a consensus problem, see \cite{sar1}. 

Many optimization methods have been extended to address synchronization and consensus problems, see \cite{jad, olf}. A key factor in the optimization methods developed for such problems is that each agent optimizes the cost function using its local variables or information, i.e. the optimization problem is in the category of decentralized optimization problems. Since synchronization cost functions depend on all agents state trajectories, a successful implementation of local optimization algorithms necessitates information exchange among the agents in the network, see \cite{sar,sar1,jad,olf,nai,law}. In this case convergence of optimization algorithms highly depend on the topology of the network.

 In this paper we employ a class of optimization methods (extremum seeking algorithms) which makes agents local optimizations independent of the state of other agents in the network. That is to say, each agent updates its current state only with respect to the monitored synchronization cost. In this setting, we accept the fact that all agents have full information about the total objective function defined for their synchronization. This problem falls within the class of cooperative team problems in which agents aim to optimize the aggregated cost function. Using the approach of this paper, agents local optimization is independent of the network topology which is one of the main contributions of this paper.

   We also consider a network of agents evolving on a connected Lie group. In this case, the convergence analysis of the proposed algorithm is obtained for a generic Riemannian metric which distinguishes our approach from the methods presented in \cite{sar, sar1, hans3}, where only the embedded Euclidean metrics were considered. The analysis presented in \cite{sar} and \cite{hans3}  is restricted to the ambient Euclidean spaces of Riemannian submanifolds. However, in general, embeddings of Riemannian manifolds may not be available (their existence is guaranteed by Nash Theorem) and the resulted Euclidean spaces may be very high dimensional. 
This makes the implementation of optimization  algorithms problematic and optimization algorithms on the main Riemannian manifolds might be more efficient in terms of computation burden, see \cite{ring}.

In terms of exposition,  Section II  presents some  mathematical preliminaries needed for the analysis of the paper and formulates the synchronization problem on Riemannian manifolds. Section III presents the extremum seeking problem for nonlinear dynamical systems on Riemannian manifolds and gives the analysis of extremum seeking systems for synchronization of agents on Riemannian manifolds. 

 \section{Preliminaries and problem formulation}

\newtheorem{definition}{Definition}
\begin{definition}[\hspace{-.01cm}\cite{Lee3}]
A Riemannian manifold $(M,g_{M})$ is a differentiable manifold $M$ together with a Riemannian metric $g_{M}$, where $g_{M}:T_{x}M\times T_{x}M\rightarrow \mathds{R}$ is symmetric and positive definite and $T_{x}M$ is the tangent space at $x\in M$ (see \cite{Lee2}, Chapter 3). For $M=\mathds{R}^{n}$, the Riemannian metric $g_{\mathds{R}^{n}}$ is given by 
 \EQ g_{\mathds{R}^{n}}\left(\frac{\partial}{\partial x_{i}},\frac{\partial}{\partial x_{j}}\right)=\delta_{ij},\hspace{.2cm}i,j=1,...,n,\nnum\EN
 where $\delta_{ij}$ is the Kronecker delta.\halmos 
\end{definition}
\begin{definition}[\hspace{-.01cm}\cite{Lee2}]
\label{d2}
For a given smooth mapping $F:M\rightarrow N$ from manifold $M$ to manifold $N$, the pushforward (differential) operator $TF$ is defined as a generalization of the Jacobian of smooth maps in Euclidean spaces as follows:
\EQ TF:TM\rightarrow TN, \EN
where
\EQ \label{puu}T_{x}F:T_{x}M\rightarrow T_{F(x)}N,\EN
and
\EQ T_{x}F(X_{x})\circ f=X_{x}(f\circ F),\hspace{.2cm}X_{x}\in T_{x}M, f\in C^{\infty}(N).\nnum\EN
\halmos\end{definition}
In this paper we present the final results for connected finite dimensional Lie groups which are manifolds equipped with smooth group operations. However, some parts of the analysis are presented for general Riemannian manifolds.
On an $n$ dimensional Riemannian manifold $M$, the length function of a smooth curve $\gamma:[a,b]\rightarrow M$ is defined as 
\EQ \ell(\gamma)=\int^{b}_{a}\big(g_{M}(\dot{\gamma}(t),\dot{\gamma}(t))\big)^{\frac{1}{2}}dt,\nnum\EN
where $g$ denotes the Riemannian metric on $M$.
The following theorem ensures that for any connected Riemannian manifold $M$, any pair of points $x,y\in M$ can be connected by a piecewise smooth path $\gamma$. \\
\newtheorem{theorem}{Theorem}
\begin{theorem}[\hspace{-.025cm}\cite{Lee3}, Page 94]
\label{t1}
Suppose $(M,g_{M})$ is an $n$ dimensional connected Riemannian manifold. Then, for any pair $x,y\in M$, there exists a piecewise smooth path which connects $x$ to $y$.\halmos
\end{theorem}
Consequently  we can define a metric (distance) $d$ on an $n$ dimensional Riemannian manifold $(M,g_{M})$ as follows:
\EQ\label{l} &&d:M\times M\rightarrow \mathds{R},\hspace{.2cm}\nnum\\&& d(x,y)=\inf_{\gamma:[a,b]\rightarrow M}\int^{b}_{a}\big(g_{M}(\dot{\gamma}(t),\dot{\gamma}(t))\big)^{\frac{1}{2}}dt, \EN where $\gamma:[a,b]\rightarrow M$ is a piecewise smooth path and $\gamma(a)=x,\gamma(b)=y$.

Employing the distance function above it can be shown that $(M,d)$ is a metric space. This is formalized by the next theorem.\\
\begin{theorem}[\hspace{-.02cm}\cite{Lee3}, Page 94]
\label{t2}
With the distance function $d$ defined in (\ref{l}), any connected Riemannian manifold is a metric space where the induced topology is the same as the manifold topology.\halmos 
\end{theorem}
The \textit{Levi-Civita} connection $\nabla:\mathfrak{X}(M)\times \mathfrak{X}(M)\rightarrow\mathfrak{X}(M)$ is the unique linear connection on $M$ (see \cite{Lee3}, Theorem 5.4)  which is torsion free and compatible with the Riemannian metric $g_{M}$ as follows ($\mathfrak{X}(M)$ is the space of smooth vector fields on $M$):  
\EQ \label{kir3}&&\mbox{compatibility with}\hspace{.2cm} g_{M}\nnum\\&&Xg_{M}(Y,Z)=g_{M}(\nabla_{X}Y,Z)+g_{M}(Y,\nabla_{X}Z),\EN
\EQ  \label{levi} \hspace{-.5cm}&&(i)(\mbox{torsion free}): \nabla_{X}Y-\nabla_{Y}X=[X,Y],\hspace{.5cm}\nnum\\\hspace{-.2cm}&&(ii): \nabla_{X}f=X(f),\EN
where $X,Y,Z\in \mathfrak{X}(M)$.

\subsection{Dynamical systems on Riemannian manifolds}
This paper focuses on dynamical systems governed by differential equations. Locally these differential equations are expressed by (see \cite{Lee2})
\EQ &&\hspace{-0cm}\dot{x}(t)=f(x(t),t),\hspace{.2cm} \nnum\\&&\hspace{-0cm}f(x(t),t)\in T_{x(t)}M,\hspace{.2cm} x(0)=x_{0}\in M, t\in[t_{0},t_{f}].\nnum\EN
The time dependent flow associated with a differentiable time dependent vector field $f$ is a map $\Phi_{f}$ satisfying :
\EQ \label{flow} &&\Phi_{f}:[t_0, t_{f}]\times [t_{0}, t_{f}]\times M\rightarrow M, \nnum\\&& (t_{0},s,x)\hookrightarrow \Phi_{f}(s,t_{0},x)\in M,\nnum\EN
and
\EQ \frac{d\Phi_{f}(s,t_{0},x)}{ds}|_{s=t}=f(\Phi_{f}(t,t_{0},x),t).\nnum\EN 
One may show, for a smooth vector field $f$, the integral flow $\Phi_{f}(s,t_{0},.):M\rightarrow M$ is a local diffeomorphism , see \cite{Lee2}.
  Here we assume that the vector field $f$ is smooth and \textit{complete}, i.e. $\Phi_{f}$ exists for all $t\in (t_{0},\infty)$.
    \subsection{Geodesic Curves}
   As known (see \cite{jost}), geodesics are defined as length minimizing curves on Riemannian manifolds which satisfy
   \EQ \nabla_{\dot{\gamma}(t)}\dot{\gamma}(t)=0,\nnum\EN
   where $\gamma(\cdot)$ is a geodesic curve on $(M,g_{M})$.
 \begin{definition}[\hspace{-.01cm}\cite{Lee3}]
 \label{def1}
 The restricted exponential map is defined by 
 \EQ \exp_{x}:T_{x}M\rightarrow M,\hspace{.2cm}\exp_{x}(v)=\gamma_{v}(1), v\in T_{x}M,\nnum\EN
 where $\gamma_{v}(1)$ is the geodesic initiating from $x$ with the velocity $v$ up to $t=1$.
 \halmos\end{definition}
 For brevity, in this paper we refer the restricted exponential maps as exponential maps. 
 For $x\in M$, consider a $\delta$ ball in $T_{x}M$ such that $B_{\delta}(0)\doteq\{v\in T_{x}M|\hspace{.2cm}||v||_{g}\doteq g_{M}(v,v)^{\frac{1}{2}}<\delta\}$. Then the geodesic ball is defined as follows.

 \begin{definition}[\hspace{-.01cm}\cite{Lee3}]
 In a neighborhood of  $x\in M$ where $\exp_{x}$ is a local diffeomorphism (this neighborhood always exists by Lemma \ref{eun} below), a geodesic ball of radius $\delta>0$ is denoted by $\exp_{x}(B_{\delta}(0))\subset M$. Also we call $\exp_{x}(\overline{B}_{\delta}(0))$ a closed geodesic ball of radius $\delta$. 
 \halmos\end{definition}
 \newtheorem{lemma}{Lemma}
 \begin{lemma}[\hspace{-.01cm}\cite{Lee3}]
 \label{eun}
 For any $x\in M$ there exists a neighborhood $B_{\delta}(0)$ in $T_{x}M$ on which $\exp_{x}$ is a diffeomorphism onto $\exp_{x}(B_{\delta}(0))\subset M$. 
 \halmos\end{lemma}
 \begin{definition}[\hspace{-.01cm}\cite{Lee3}]
 A normal neighborhood around $x\in M$ is any open neighborhood  of $x$ which is a diffeomorphic image of a star shaped neighborhood of $0\in T_{x}M$ under $\exp_{x}$ map.
 \halmos\end{definition}

 \begin{definition}\label{inj}The injectivity radius of $M$ is  
 \EQ i(M)\doteq \inf_{x\in M}i(x),\nnum\EN
 where
\EQ&&\hspace{-.8cm}i(x) \doteq \sup\{ r\in\mathds{R}_{\ge0}| \exp_x \mbox{is differmorphic onto} \exp_x B_r(0)\}.\nnum\\\EN
\halmos\end{definition}  

The following lemma displays a relationship between normal neighborhoods and metric balls defined before on $M$.
\begin{lemma}[\hspace{-.01cm}\cite{Pet}]
\label{lpp}
If $\exp_{x}(\cdot),\hspace{.2cm}x\in M$, is a local diffeomorphism on $B_{\epsilon}(0)\subset T_{x}M,\hspace{.2cm}\epsilon\in \mathds{R}_{>0}$, and $B(x,r)\subset \exp_{x}B_{\epsilon}(0)$, then 
\EQ \exp_{x}B_{r}(0)=B(x,r),\nnum\EN
where $B(x,r)$ is the metric ball with respect to the Riemannian distance function. 
\hspace*{7.5cm}\halmos\end{lemma}
We note that $B_{\epsilon}(0)$ is the metric ball of radius $\epsilon$ with respect to the Riemannian metric $g_{M}$ in $T_{x}M$.

The following lemma bounds the injectivity radius of compact Riemannian manifolds, see Definition \ref{inj}.
 
\begin{lemma}[\hspace{-.01cm}\cite{Klin}]
\label{kl}
The injectivity radius $i(x), x\in M$ is continuous with respect to $x$ and is bounded from below for compact Riemannian manifolds.
\halmos\end{lemma}

By the results of \cite{Pet}, Corollary 5.3, and Lemma \ref{kl}, in the case $i(M)>0$, for any $r\leq i(M),$ such that $B(x,r)\subset \exp_{x}B_{i(M)}(0)$, we have
\EQ B(x,r)=\exp_{x}B_{r}(0).\nnum\EN
\section{Synchronization on Riemannian manifolds and Lie groups}
Let us consider a set of $m$ agents $\mathcal{A}\doteq \{1,\cdots, m\}$, on a connected  $n$ dimensional Riemannian manifold $(M,g_{M})$ where the state of each agent lies on $M$, i.e.  $x_{i}\in M,\hspace{.2cm}i=1,\cdots m$. The synchronization for $\mathcal{A}$ is met when $x_{1}=x_{2}=\cdots x_{m}\in M$, see \cite{sar1}. For the network of agents $\mathcal{A}$, an undirected graph $\mathcal{G}(\mathcal{V},\mathcal{E})$ has a finite set of vertices $\mathcal{V}$ and a set of unordered edges $\mathcal{E}$. A link which connects vertices $i$ and $j$ is denoted by $(i,j)\in \mathcal{E}$. Corresponding to  $x_{i},\hspace{.2cm}i=1,\cdots,m$, a cost function to penalize the deviation from the synchronized configuration is proposed in \cite{sar1} as  
\EQ \label{cost}J(x_{1},\cdots,x_{m})=\frac{1}{2}\sum_{(i,j)\in \mathcal{E}}d^{2}(x_{i},x_{j}),\EN
where $d$ is the Riemannian metric on $(M,g_{M})$. As is obvious the unique global minimum of $J$ is given by $x_{1}=x_{2}=\cdots x_{m}$, i.e. at the synchronization state.   The optimization problem defined in (\ref{cost}), is a special case of the optimization of a cost function $J:M\times M\cdots \times M\rightarrow \mathds{R}_{\geq 0}$, defined on the Riemannian manifold $M\times M\cdots \times M$. In the case that the graph $\mathcal{G}$ is fully connected the cost function (\ref{cost}) changes to $J(x_{1},\cdots,x_{m})=\frac{1}{2}\sum^{m}_{i=1}\sum^{m}_{j=1,j\ne i}d^{2}(x_{i},x_{j})$. As an example in the case $M=SO(n)$ and $d$ as the Frobenius metric, we have \cite{sar1}
\EQ \label{cost1}J(x_{1},\cdots,x_{m})&=&\frac{1}{2}\sum_{(i,j)\in \mathcal{E}}tr\big((x_{i}-x_{j})^{T}(x_{i}-x_{j})\big)\nnum\\&=&\sum_{(i,j)\in \mathcal{E}}(n-tr(x^{T}_{i}x_{j})),\EN
where $x_{i}\in SO(n)$. 
One of the most popular optimization algorithms for minimization(maximization) of $J$ is the gradient descent method. The decentralized version of the gradient method for each agent is given by
\EQ \dot{x}_{i}=-grad_{x_{i}}J,\nnum\EN
where $dJ_{i}(X)=g(grad_{x_{i}},X)$ for all $X\in T_{x}M$. Note that $dJ_{i}:T_{x}M\rightarrow \mathds{R}$, is the differential form of $J$ with respect to the state of agent $i$. For the cost function $J=\sum_{(i,j)\in \mathcal{E}}(n-tr(x^{T}_{i}x_{j}))$ on $SO(n)^{m}$, we have
\EQ\label{gr} \dot{x}_{i}=\frac{1}{2}x_{i}\sum_{j:(i,j)\in \mathcal{E}}(x^{T}_{i}x_{j}-x^{T}_{j}x_{i}),\hspace{.2cm}i=1,\cdots,m.\EN
Obviously, in order to implement (\ref{gr}), agent $i$ has to know the state $x_{j}$ of all the agents $j$, where $(i,j)\in \mathcal{E}$. That necessitates  communication or information exchange among agents  in decentralized algorithms, see \cite{jad, olf}.

In this paper we present the final results on a connected Lie group $G$ (for the definition of Lie groups see \cite{varad}).  Let us denote a Lie group  $(G,\star)$ as the state configuration manifold for all agents. Note that $\star$ is the group operation of $G$. 
We recall that the Lie algebra $\mathcal{L}$ of a Lie group $G$ (see \cite{Lewis},\cite{ varad}) is the tangent space at the identity element $e$ with the associated Lie bracket defined on the tangent space of $G$, i.e.  $\mathcal{L}=T_{e}G$. 
A vector field $X$ on $G$ is \textit{left invariant } if
\EQ \forall g_{1},g_{2}\in G,\quad X(g_{1}\star g_{2})=T_{g_{2}}g_{1}X(g_{2}),\nnum\EN
where $g\star:G\rightarrow G,\hspace{.2cm}g\star(h)=g\star h,\hspace{.2cm}, T_{g_{2}}g:T_{g_{2}}G\rightarrow T_{g\star g_{2}}G$. That immediately implies $X(g\star e)=X(g)=T_{e}L_{g}X(e)$.
For a left invariant vector field $X$, we define the \textit{exponential map} as
\EQ \label{kirkir1}\exp:\mathcal{L}\rightarrow G,\quad \exp(tX(e)):=\Phi(t,X),t\in \mathds{R},\EN

where $\Phi(t,X)$ is the integral flow of $\dot{g}(t)=X(g(t))$ with the boundary condition $g(0)=e$. Note that $\exp(tX(e))$ is not necessarily the same as the geodesic $\exp $ map defined in Definition \ref{def1}.

The synchronization cost defined in (\ref{cost}) has a critical set of points denoted by $G_{c}$, which minimizes (\ref{cost}) , where $g\in G_{c}$ implies that $J(g_{c})=0$. Motivated by the cost function  (\ref{cost}), the synchronization critical set is given by
\EQ \label{gc}G_{c}\doteq \{(g,\cdots,g)\,|\,g\in G\}.\EN
It is imediate that $G_{c}\subset G\times \cdots\times G$. 
The set (\ref{gc}) characterizes all the possibilities of synchronization for the agents in the network. The following lemma shows that $G_{c}$ is a Lie group in the topology induced by $G^{m}\doteq G\times \cdots\times G$. We denote the Lie algebra corresponding to $G^{m}$ by $\mathcal{L}^{m}$.

\begin{lemma}
\label{l1}
The critical set (\ref{gc}) is a Lie subgroup of the Lie group $G^{m}$.
\end{lemma}
\begin{proof}
First we note that $G^{m}$ is a Lie group. This is immediate by the group structure of $G$. In order to prove $G_{c}$ is a Lie subgroup we have to show that $G_{c}$ is a closed subgroup of $G^{m}$. Obviously $G_{c}$ is not empty. For any  $g_{1},g_{2}\in G_{c}$ we have $g_{1}\overline{\star} g_{2}^{-1}\in G_{c}$, where $\overline{\star}$ is the group operation inherited from $G$ on $G^{m}$. This proves that $G_{c}$ is a subgroup of $G^{m}$. In order to show $G_{c}$ is a Lie group it remains to prove the closeness of $G_{c}$ in the topology of $G^{m}$. This is also immediate since based on the structure of $G_{c}$, any converging sequence $g_{n}\rightarrow g^{*},\hspace{.2cm}g_{n}\in G_{c}$ implies that $g^{*}\in G_{c}$, which yields the closeness of $G_{c}$ in the topology of $G^{m}$. By applying Cartan's Lemma \cite{varad}, $G_{c}$ is a Lie subgroup and consequently a Lie group.    
\end{proof}

The following lemma gives a Riemannian structure on the Lie group $G^{m}$ based on inner products in $\mathcal{L}^{m}$.
\begin{lemma}[\hspace{-.01cm}\cite{Lewis}]
An inner product $\mathds{I}:\mathcal{L}^{m}\times \mathcal{L}^{m}\rightarrow \mathds{R}$, induces a left invariant Riemannian metric on $G^{m}$ as 
\EQ &&g_{G^{m}}(v_{g},w_{g})=\mathds{I}(T_{g}g^{-1}(v_{g}),T_{g}g^{-1}(w_{g})),\nnum\\&&\hspace{.2cm}v_{g},w_{g}\in T_{g}G^{m},\nnum\EN
where $T_{g}g^{-1}(v_{g})\in \mathcal{L}^{m}$. Furthermore, any left invariant Riemannian metric $g_{G^{m}}$ is identified via left translation by its value $g_{g^{m}}(e^{m})$, where $e^{m}$ is the identity element of $G^{m}$.

\halmos\end{lemma}

Note that $Tg^{-1}$ is the pushforward of the smooth map $g^{-1}\overline{\star}:G^{m}\rightarrow G^{m}$.

In order to analyze the extremum seeking algorithm in the next section for the synchronization cost (\ref{cost}), we need to use the notion of \textit{Quotient Manifolds} as follows.
	As shown by Lemma \ref{l1}, the critical set $G_{c}$ is a Lie subgroup of $G^{m}$. This implies the existence of a left (right) group action from $G_{c}$ to $G^{m}$ by 
	\EQ G_{c}\times G^{m}\rightarrow G^{m},\nnum\EN
	where $g_{1}\overline{\star}g_{2}\in G^{m}$ for $g_{1}\in G_{c}$ and $g_{2}\in G^{m}$. By employing the left action above, we introduce an equivalent class induced by $G_{c}$ on $G^{m}$ as $ g_{1}\sim g_{2}, g_{1},g_{2}\in G^{m}$, if there exists $\hat{g}\in G_{c}$, such that $\hat{g}\overline{\star}g_{1}=g_{2}$. This induces a projection $\pi:G^{m}\rightarrow G^{m}/G_{c}$ by $\pi(g)=[g]$, where $[\cdot]$ is the equivalent class operator and $G^{m}/G_{c}$ is the quotient space. In this paper we denote $[g]\subset G^{m}$ and $\pi	(g)\in G^{m}/G_{c}$. 
	Note that in general, quotient spaces are not manifolds and may not be even Housdorff spaces, see \cite{Lee2}.  Since by Lemma \ref{l1} $G_{c}$ is a Lie group, the following result shows that the quotient space $G^{m}/G_{c}$ is a smooth manifold.
	\begin{lemma}[\hspace{-.01cm}Theorem 21.17 in \cite{Lee2}, edition 2012]
	\label{l2}
	The quotient space $G^{m}/G_{c}$ has a smooth manifold structure and $\pi:G^{m}\rightarrow G^{m}/G_{c}$ is a smooth submersion.
 	\end{lemma}
	
	\begin{definition}
	A smooth function $J:G^{m}\rightarrow \mathds{R}$ is invariant with respect to $G_{c}$ if 
	\EQ J(g_{c}\overline{\star}g)=J(g),\hspace{.2cm}g_{c}\in G_{c},g\in G^{m}.\nnum\EN
	\halmos\end{definition}
	As an example consider $(G^{m},\overline{\star})=(\mathds{R}^{m},+)$. Hence, $G_{c}=\bigcup_{r\in\mathds{R}}(r,\cdots,r)\subset \mathds{R}^{m}$. The Euclidean distance function $||\cdot||^{2}_{\mathds{R}^{m}}$ is invariant with respect to $G_{c}$ since
	\EQ ||r_{1}-r_{2}||^{2}_{\mathds{R}^{m}}=||r_{1}+r-r_{2}-r||^{2}_{\mathds{R}^{m}}.\nnum\EN
	Another example is the cost function (\ref{cost1}) which is invariant with respect to its corresponding $G_{c}$ since $tr(x^{T}_{i}Q^{T}Qx_{j})=tr(x^{T}_{i}x_{j})$ for $Q\in SO(n)$.
	\begin{definition}
	\label{def2}
	On the Lie group $G^{m}$, a vector field $X\in \mathfrak{X}(G^{m})$ is invariant with respect to $G_{c}$, if
	\EQ X(g_{c}\overline{\star}g)=T_{g}g_{c}(X(g)),\hspace{.2cm}g_{c}\in G_{c},g\in G^{m},\nnum\EN
	where $Tg_{c}$ is the pushforward of the smooth group operation  $g_{c}\overline{\star}:G^{m}\rightarrow G^{m}$. 
	\halmos\end{definition}
	We note that vector fields on the base manifold $G^{m}$ do not necessarily induce vector fields on $G^{m}/G_{c}$. This is due to the fact that if $X(g_{1})\ne X(g_{2}),\hspace{.2cm}g_{1},g_{2}\in [g]$, where $X\in \mathfrak{X}(G^{m})$, then $T_{g_{1}}\pi(X(g_{1}))\in T_{\pi(g_{1})}G^{m}/G_{c}$ is not necessarily identical to $T_{g_{2}}\pi(X(g_{2}))\in T_{\pi(g_{2})}G^{m}/G_{c}$. However, in the case that $X$ is invariant with respect to $G_{c}$, $T\pi(X)$ induces a vector field on $G^{m}/G_{c}$. This follows as
	\EQ\label{kir} \pi\circ g_{c}\overline{\star}g=\pi\circ g,\hspace{.2cm} g_{c}\in G_{c},g\in G^{m}.\EN
	Hence, the smoothness of $\pi$ (Lemma \ref{l2}), implies
	\EQ T_{g_{c}\overline{\star}g}\pi\circ T_{g}g_{c}(X(g))&=&T_{ g_{c} \overline{\star}g}\pi\circ X(g_{c}\overline{\star}g)\nnum\\&=&T_{g}\pi \circ X(g),\nnum\EN
	where the first equality is by Definition \ref{def2} and the second equality is given by (\ref{kir}).
	This implies that both $X(g_{c}\overline{\star}g)$ and $X(g)$ induce the same tangent vector at $\pi(g)\in G^{m}/G_{c}$.

	Parallel to the construction of vector fields on $G^{m}/G_{c}$, we can assign a Riemannian metric to $G^{m}/G_{c}$. 
	It is important to note that the structure of the Riemannian metric of the base manifold $G^{m}$ stipulates the structure of the Riemannian metric in $G^{m}/G_{c}$. Following the results of \cite{abs}, chapter 3, for any tangent vector $v_{g}\in T_{g}G^{m}/G_{c}$ and for any $\hat{g}\in\pi^{-1}(g)$, there exist  tangent vectors $\hat{v}_{\hat{g}}\in T_{\hat{g}}G^{m}$  such that
	\EQ\label{hl} T_{\hat{g}}\pi(\hat{v}_{\hat{g}})=v_{g}\in T_{g}G^{m}/G_{c}.\EN

	In order to define a Riemannian metric on $G^{m}/G_{c}$, we need to employ the Riemannian metric of $G^{m}$ and apply that to horizontal lifts (see \cite{abs}) of tangent vectors at $TG^{m}/G_{c}$. However, for $\hat{g}_{1},\hat{g}_{2}\in \pi^{-1}(g),\hspace{.2cm}\hat{v}_{\hat{g}_{1}},\hat{w}_{\hat{g}_{1}}\in T_{\hat{g}_{1}}G^{m}, \hat{v}_{\hat{g}_{2}},\hat{w}_{\hat{g}_{2}}\in T_{\hat{g}_{2}}G^{m}$, it is not guaranteed that $g_{G^{m}}(\hat{v}_{\hat{g}_{1}},\hat{w}_{\hat{g}_{1}})=g_{G^{m}}(\hat{v}_{\hat{g}_{2}},\hat{w}_{\hat{g}_{2}})$, whereas $T_{\hat{g}_{1}}\pi(\hat{v}_{\hat{g}_{1}})=T_{\hat{g}_{2}}\pi(\hat{v}_{\hat{g}_{2}})$ and $T_{\hat{g}_{1}}\pi(\hat{w}_{\hat{g}_{1}})=T_{\hat{g}_{2}}\pi(\hat{w}_{\hat{g}_{2}})$.
	The following lemma shows that in the case that the Riemannian metric of the base manifold $G^{m}$ comes from an inner product on its Lie algebra then we can define an unambiguous Riemmanina metric on $G^{m}/G_{c}$ with respect to $g_{G^{m}}$.
	\begin{lemma}
	\label{l3}
	Consider the Lie group $G^{m}$ with a Riemannian metric corresponding to an inner product $\mathds{I}:\mathcal{L}^{m}\times \mathcal{L}^{m}\rightarrow \mathds{R}$. Then, for all $\hat{g}_{1},\hat{g}_{2}\in [g],\hspace{.2cm}\hat{v}_{\hat{g}_{1}},\hat{w}_{\hat{g}_{1}}\in T_{\hat{g}_{1}}G^{m}, \hat{v}_{\hat{g}_{2}},\hat{w}_{\hat{g}_{2}}\in T_{\hat{g}_{2}}G^{m}$, where $T_{\hat{g}_{1}}\pi(\hat{v}_{\hat{g}_{1}})=T_{\hat{g}_{2}}\pi(\hat{v}_{\hat{g}_{2}})$ and $T_{\hat{g}_{1}}\pi(\hat{w}_{\hat{g}_{1}})=T_{\hat{g}_{2}}\pi(\hat{w}_{\hat{g}_{2}})$, we have
	\EQ g_{G^{m}}(\hat{v}_{\hat{g}_{1}},\hat{w}_{\hat{g}_{1}})=g_{G^{m}}(\hat{v}_{\hat{g}_{2}},\hat{w}_{\hat{g}_{2}}).\nnum\EN
	\end{lemma}
	\begin{proof}
	We need to show that $G_{c}$ acts as a group of isometries with respect to the Riemannian metric induced by $\mathds{I}$. Since $\hat{g}_{1},\hat{g}_{2}\in [g]$, there exists $g_{c}\in G_{c}$ such that $\hat{g}_{2}=g_{c}\overline{\star}g_{1}$. Then it is sufficient to show
	
	\EQ &&\hspace{-1cm}\mathds{I}(T_{\hat{g}_{1}}\hat{g}^{-1}_{1}(\hat{v}_{\hat{g}_{1}}),T_{\hat{g}_{1}}\hat{g}^{-1}_{1}(\hat{w}_{\hat{g}_{1}}))=\nnum\\&&\hspace{-1cm}\mathds{I}(T_{g_{c}\overline{\star}\hat{g}_{1}}(g_{c}\overline{\star}\hat{g}_{1})^{-1}T_{\hat{g}_{1}}g_{c}(\hat{v}_{\hat{g}_{1}}),T_{g_{c}\overline{\star}\hat{g}_{1}}(g_{c}\overline{\star}\hat{g}_{1})^{-1}T_{\hat{g}_{1}}g_{c}(\hat{w}_{\hat{g}_{1}})).\nnum\\\EN
	Hence, the statement holds if  
	\EQ T_{\hat{g}_{1}}\hat{g}^{-1}_{1}(\hat{v}_{\hat{g}_{1}})=T_{g_{c}\overline{\star}\hat{g}_{1}}(g_{c}\overline{\star}\hat{g}_{1})^{-1}T_{\hat{g}_{1}}g_{c}(\hat{v}_{\hat{g}_{1}}),\nnum\EN
	which is trivial since $\hat{g}^{-1}_{1}\overline{\star}g^{-1}_{c}\overline{\star} g_{c}\overline{\star}\hat{g}_{1}=\hat{g}^{-1}_{1}\overline{\star}\hat{g}_{1}.$
	The statment above proves that $G_{c}$ acts isometrically on $G^{m}$ and the proof is complete.
	\end{proof}
	By employing the results of Lemma \ref{l3} we define the induced Riemannian metric on the quotient manifold $G^{m}/G_{c}$ as follows.
	For tangent vectors $v_{g},w_{g}\in T_{g}G^{m}/G_{c}$ define $g_{G^{m}/G_{c}}(v_{g},w_{g})\doteq g_{G^{m}}(\hat{v}_{[g]},\hat{w}_{[g]}),$
	where $\hat{v}_{[g]},\hat{w}_{[g]}$ are the unique horizontally lifted tangent vectors in $T_{[g]}G^{m}$ corresponding to $v_{g}$ and $w_{g}$, see \cite{abs}. Note that $[g]$ is a subset of $G^{m}$ and not necessarily a single point in $G^{m}$. However, with no further confusion we accept the notation $T_{[g]}G^{m}$ as the tangent space of all elements in $[g]$ on $G^{m}$. 
	
	Any smooth invariant function $J:G^{m}\rightarrow \mathds{R}$ induces a smooth function $\hat{J}:G^{m}/G_{c}\rightarrow \mathds{R}$ such that $J=\hat{J}\circ \pi$.
	The horizontal lift of the gradient of $\hat{J}$ in $G^{m}/G_{c}$ is the gradient of $J$ in $G^{m}$. That is to say $\overset{H}{grad}_{g}\hat{J}=grad_{[g]}J,\hspace{.2cm}g\in G^{m}/G_{c},$
	where $\overset{H}{\cdot}$ gives the horizontal lift of the tangent vector $grad_{g} \hat{J}\in T_{g}G^{m}/G_{c}$. For detailed discussion on the horizontal lift and the equivalence of the gradients in the base manifold and the quotient manifold see \cite{abs}, chapter 3. 
	
\section{Extremum seeking algorithm for synchronization}
An extremum seeking closed loop is shown in Figure \ref{12}. This is the simplest form of the extremum seeking algorithm to minimize/maximize a scalar function $J:\mathds{R}\rightarrow \mathds{R}$. 
\begin{figure}
\begin{center}
\vspace{0cm}
\hspace*{-1cm}
\includegraphics[scale=.3]{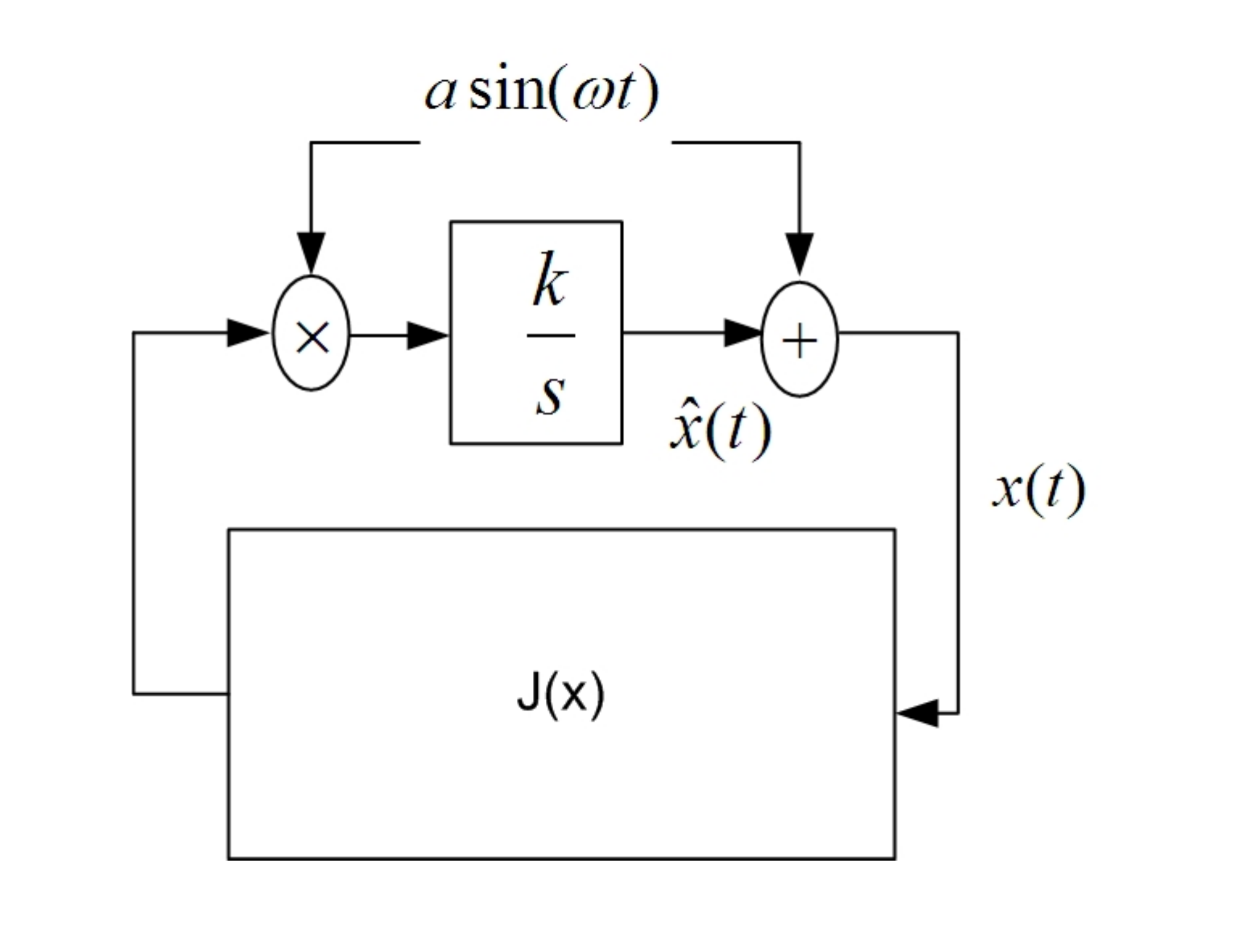}
 \caption{  Extremum seeking closed loop system }
      \label{12}
      \end{center}
   \end{figure} 
   The dither signal $a\sin(\omega t)$ provides a variation of the searching signal $\hat{x}(t)$ in the one dimensional space $\mathds{R}$. 
  The output of the integrator is $\hat{x}$, where $x=\hat{x}+a\sin(\omega t)$. The dynamical equations in $\hat{x}$ coordinates are given by
  \EQ\label{kooni} \dot{\hat{x}}(t)&=&k a\sin(\omega t)J(\hat{x}(t)+a\sin(\omega t)),\EN
	where without loss of generality we assume $k=-1$.
  \subsection{Averaging of the synchronization vector field} 
	
	As known \cite{tan}, on average, the dynamical system (\ref{kooni}) behaves as a gradient algorithm. Under technical assumptions for the dither signal, it is guaranteed that the state trajectory $\hat{x}$ converges to a neighborhood of an optimizer point of $J$, see \cite{krs1,tan}. This neighborhood is shrunken by adjusting the magnitude and the frequency of the dither signal, see \cite{krs1, tan}. 
	
	Consider an $n$ dimensional Riemannian manifold $(M,g_{M})$. For any $x\in M$, we consider the following local time-varying perturbation 
	(\textit{geodesic dither})
  \EQ \label{gd}x_{p}(t)=\exp_{x}\sum^{n}_{i=1}a_{i}\sin(\omega_{i}t)\frac{\partial}{\partial x_{i}},\hspace{.2cm}0<a_{i},\EN 
  where $\frac{\partial}{\partial x_{i}},\hspace{.2cm}i=1,\cdots,n$, are the basis for the tangent space at $x$. As defined  before, $\exp_{x}v,\hspace{.2cm}v\in T_{x}M$ is a geodesic emanating from $x\in M$ with velocity $v$. In this case we perturbed different coordinates on $M$ with different frequencies $\omega_{i},\hspace{.2cm}i=1,\cdots,n$.
	The optimization of $J$ on $M$ is carried out by the state trajectory of the following time-varying vector field on $M$.
	\EQ \label{kk}&&\hspace{-.5cm}f(\hat{x},t)\doteq - \sum^{n}_{i=1}a_{i}\sin(\omega_{i}t)J(\exp_{\hat{x}}\sum^{n}_{i=1}a_{i}\sin(\omega_{i}t)\frac{\partial}{\partial x_{i}})\frac{\partial}{\partial x_{i}},\nnum\\&&\hspace{-.5cm}f(\hat{x},t)\in T_{\hat{x}}M,\hspace{.2cm}i=1,\cdots,n,\EN
  where the optimizing trajectory $\hat{x}(\cdot)$ is a solution of the time dependent differential equation $\dot{\hat{x}}(t)=f(\hat{x},t)\in T_{\hat{x}}M$. Note that the magnitudes $a_{i},\hspace{.2cm}i=1,\cdots,n,$ are selected such that $||\sum^{n}_{i=1}a_{i}\sin(\omega_{i}t)\frac{\partial}{\partial x_{i}}||_{g}<i(M)$. 
	
	This algorithm is generalized for the synchronization problem defined above. In this case each agent updates its state via
	\EQ \label{grg}&&\hspace{-.5cm}\dot{\hat{x}}_{j}=-\sum^{n}_{i=1}a^{j}_{i}\sin(\omega_{i}t)\times \nnum\\&&\hspace{-.5cm}J(\cdots,\exp_{\hat{x}_{j}}\sum^{n}_{i=1}a^{j}_{i}\sin(\omega^{j}_{i}t)\frac{\partial}{\partial x_{i}},\cdots)\frac{\partial}{\partial x_{i}},\hspace{.2cm}j=1,\cdots,m,\nnum\\\EN
	and the state of agent $j$ is computed by 
	\EQ x_{j}(t)=\exp_{\hat{x}_{j}(t)}\sum^{n}_{i=1}a^{j}_{i}\sin(\omega^{j}_{i}t)\frac{\partial}{\partial x_{i}},\hspace{.2cm}j=1,\cdots,m.\nnum\EN
	\newtheorem{remark}{Remark}
	Note that in this algorithm the cost function $J$ in (\ref{grg}) contains all the perturbations induced by all agents.
	
	
	\begin{remark}As is obvious, the optimization algorithm (\ref{grg}) requires only information about the cost function $J$ at each time $t$. This makes the implementation of the decentralized algorithm independent of the state of other agents and consequently from the topology of the network. However, it is required that all agents have access to the synchronization cost at all time. 
	\halmos\end{remark}
	\begin{remark}We note that (\ref{grg}) is formulated with respect to a generic Riemannian metric $g_{M}$ and does not depend upon the embedding Euclidean space. This is a major distinction between our method and the algorithms presented in \cite{sar,sar1,hans3}. The choice of the Riemannian metric $g_{M}$  affects the entire geodesic curves on $M$ and different metrics will result in distinct optimization trajectories.  \halmos\end{remark}
	The algorithm presented in (\ref{grg}) is developed for optimization of cost functions on Riemannian manifolds and with no modification can be employed for optimization on Lie groups.
	We accept the following assumptions for the synchronization algorithm and the cost function on the Lie group $G^{m}$ introduced before.
	\newtheorem{assumption}{Assumption}
  \begin{assumption}
	\label{as}
	(i): Assume $\omega^{j}_{i}=\omega \bar{\omega}^{j}_{i},\hspace{.2cm}\omega>0$, where the frequencies $ \bar{\omega}^{j}_{i}, i=1,\cdots,n, j=1,\cdots,m$ are distinct, rational and not combination of each other as   $\bar{\omega}^{j}_{i}\ne\bar{\omega}^{k}_{l}$,  $2\bar{\omega}^{j}_{i}\ne \bar{\omega}^{k}_{l}$ and $\bar{\omega}^{j}_{i}\ne \bar{\omega}^{k}_{l}+\bar{\omega}^{h}_{d}$ for distinct $i,l,d\in1,\cdots,n$ and $j,k,h\in 1,\cdots,m$. \\
	(ii): The synchronization cost function $J:G^{m}\rightarrow \mathds{R}$ is smooth and invariant with respect to $G_{c}$. Also $J(g)$ is minimized if and only if $g\in G_{c}$. 
	
	\halmos\end{assumption}
	\begin{remark}
	We note that in order to implement an extemum seeking algorithm, we need the uniqueness of a local minimum (maximum) for the cost function of  interest, see \cite{krs1, tan}. This condition is obviously violated since the set  $G_{c}$ given in (\ref{gc}) does not posses such a property. As is obvious $G_{c}$ is a closed connected subset of $G^{m}$ which is not necessarily compact. 
	\halmos\end{remark}
	
	The following lemma formulates the average behavior of the synchronization algorithm (\ref{grg}) on a Riemannian manifold $(M,g_{M})$. Note that we present the results for general Riemannian manifolds where results directly apply to $G^{m}$ which is the manifold of interest in this paper. 
	\begin{lemma}
	\label{l4}
	Consider the synchronization algorithm (\ref{grg}) for agents moving on an $n$ dimensional Riemannian manifold $(M,g_{M})$. Then, subject to Assumption \ref{as}, at each $x^{m}\in M^{m}$ and for each agent $j$, the average vector field of (\ref{grg})  is in a perturbation form as $\sum^{n}_{i=1}\nabla_{\frac{\partial}{\partial x_{i}}}J(\cdots,x_{j},\cdots)\frac{\partial}{\partial x_{i}}+O(\max_{i\in\{ 1,\cdots,n\},j\in \{1,\cdots,m\}}|a^{j}_{i}|^{4})$, where $\nabla$ is the Levi-Civita connection of $(M,g_{M})$.
	\halmos\end{lemma}
	In order to give the proof of Lemma \ref{l4} we need to study the lifting of vector fields in product manifolds and relate the Levi-Civita connections of embedded submanifolds. 
	As is obvious, the product space $M^{m}\doteq M\times\cdots\times M$ is a smooth manifold provided $M$ is a smooth manifold. We consider the product Riemannian metric on $M^{m}$ which is given by $g_{M^{m}}=g_{M}\oplus\cdots\oplus g_{M}$. Note that there exists an inclusion embedding $\iota:M\rightarrow M^{m}$ which is smooth in the product topology of $M^{m}$, see \cite{Lee2}, Chapter 7.  Now assume $X\in \mathfrak{X}(M)$ then it is possible to locally extend $X$ to a vector field $\bar{X}$ on $M^{m}$. However, this extension is not unique and for any $x \in M$, $\bar{X}$ and $X$ need to agree on $T_{x}M$, i.e. $X_{x}=\bar{X}^{T}_{x}$, where $\cdot^{T}$ is the projection operator on $T_{x}M$ with respect to $g_{M^{m}}$.  Let us denote the local coordinates around $x\in M$ by $(x_{1},\cdots,x_{n})$. Hence, locally $X$ is given by $\sum^{n}_{i=1}X_{i}(x)\frac{\partial}{\partial x_{i}}$, where $X_{i}$ are smooth functions, see  \cite{Lee2}, Lemma 4.2. By employing the local coordinates of the product manifold we may write the local coordinates of $\iota(x)\in M^{m}$ by $(x_{1},\cdots , x_{n}, y_{1},\cdots, y_{n(m-1)})$, where $y_{i}$ are local coordinates induced from $M^{m-1}$.  In this case we consider a particular local  extension of $X$ given by 
	
	\EQ \label{le}\bar{X}=\sum^{n}_{i}X_{i}\frac{\partial}{\partial x_{i}}+\sum ^{n(m-1)}_{i=1}0\frac{\partial}{\partial y_{i}},\EN
	around $\iota(x)\in M^{m}$. This defines a local vector field on $M^{m}$. The product metric of $M^{m}$ guarantees that $g_{M}(X,Y)=g_{M^{m}}(\bar{X},\bar{Y})$, i.e. $\iota$ is an isometric embedding. 
	
	Corresponding to the product metric, we have the Levi-Civita connection $\overset{p}{\nabla}$ on $M^{m}$. The Gauss formula , see \cite{Lee3}, Theorem 8.2 gives the following  relationship between $\overset{p}{\nabla}$ and $\nabla$ on $M^{m}$ and $M$ respectively.
	\begin{lemma}[Guass formula, \cite{Lee3}, Theorem 8.2]
	If $X,Y\in \mathfrak{X}(M)$ are arbitrarily extended to $\bar{X},\bar{Y}\in \mathfrak{X}(M^{m})$, then 
	\EQ \overset{p}{\nabla}_{\bar{X}}\bar{Y}=\nabla_{X}Y+II(X,Y),\nnum\EN
	where $II(\cdot,\cdot)$ is the second fundamental form of $M$.  
	\halmos\end{lemma}
	The following lemma shows that for the particular extension (\ref{le}) the second fundamental form $II$ vanishes.
	\begin{lemma}[\hspace{-.01cm}\cite{car}, exercise 6.1]
	\label{lee}
	If $X,Y\in \mathfrak{X}(M)$ are locally extended to $\bar{X},\bar{Y}\in \mathfrak{X}(M^{m})$ based on (\ref{le}), then 
	\EQ \overset{p}{\nabla}_{\bar{X}}\bar{Y}=\nabla_{X}Y.\nnum\EN
	\halmos\end{lemma}
	By employing the results of Lemma \ref{lee}, with no further confusion, we only consider the extension of vector fields given in (\ref{le}) and denote the extended vector field $\bar{X}$ by $X$. The proof of Lemma \ref{l4} is given as follows.

	\begin{proof}(Lemma \ref{l4})
	 
	The proof is based on the Riemannian structure of product spaces and the Taylor expansion of $J$ on $M^{m}$.  By employing the results of Lemma \ref{lee}, we have the following decomposition for $\overset{p}{\nabla}$.
	\EQ \label{pr}\overset{p}{\nabla}_{\sum_{j=1}^{m} Y_{j}}\sum^{m}_{j=1}X_{j}=\sum^{m}_{j=1}\nabla_{Y_{j}}X_{j},\EN
	where $Y_{j},X_{j}\in \mathfrak{X}(M)$ and $\nabla$ is the Levi-Civita connection of $(M,g_{M})$. Note that $\sum^{m}_{j=1}X_{j},\sum^{m}_{j=1}Y_{j}$ are vector fields on $M^{m}$ with respect to the extension introduced in (\ref{le}). Based on (\ref{grg}), each agent $j$ perturbs its current state $\hat{x}_{j}$ by the geodesic $\exp_{\hat{x}_{j}}\sum^{n}_{i=1}a^{j}_{i}\sin(\omega^{j}_{i}t)\frac{\partial}{\partial x_{i}}$ which is the evaluation of the geodesic curve $\gamma_{j}(\theta,t)\doteq \exp_{\hat{x}_{j}}\theta\sum^{n}_{i=1}a^{j}_{i}\sin(\omega^{j}_{i}t)\frac{\partial}{\partial x_{i}}$ at $\theta=1$. The geodesic curves $\gamma_{j}(\theta,t), j=1,\cdots,m$ induce the curve $\Gamma(\theta,t)\doteq (\gamma_{1}(\theta,t),\cdots,\gamma_{m}(\theta,t))\in M^{m}$. Note that $\theta$ is the parametrization of the geodesic $\gamma_{j}$ on $(M,g_{M})$ and $t$ appears as a parameter in the vector field which generates $\gamma_{j}$. By the properties of $\overset{p}{\nabla}$ given in (\ref{pr}), we have 
	\EQ \overset{p}{\nabla}_{\dot{\Gamma}(\theta,t)}\dot{\Gamma}(\theta,t)=\sum^{m}_{j=1}\nabla_{\dot{\gamma_{j}}(\theta,t)}\dot{\gamma_{j}}(\theta,t)=0,\nnum\EN
	where $\dot{\gamma_{j}}(\theta,t)=\frac{d\gamma}{d\theta}$ and $\nabla_{\dot{\gamma_{j}}(\theta,t)}\dot{\gamma_{j}}(\theta,t)=0$ since $\gamma_{j}$ are geodesics on $(M,g_{M})$. This implies that $\Gamma$ is a geodesic on $M^{m}$. This is due to the special product metric of $M^{m}$ and is not necessarily true for all Riemannian metrics on $M^{m}$. 
	
	Then the Taylor expansion of $J$ along the geodesic $\exp_{x^{m}}\theta X$, where $x^{m}=(x_{1},\cdots,x_{m})\in M^{m}, X\in T_{x^{m}}M^{m}$, is given by (see \cite{smith})
  \EQ \label{tay}&&\hspace{-.8cm}J(\exp_{x^{m}}\theta X)=J(x^{m})+\theta (\overset{p}{\nabla}_{X}J)(x^{m})+...+\frac{\theta^{k-1}}{(n-1)!}\times\nnum\\&&\hspace{-0.8cm}(\overset{p}{\nabla}^{k-1}_{X}J)(x^{m})+\frac{\theta^{k}}{(k-1)!}\int^{1}_{0}(1-s)^{k-1}\overset{p}{\nabla}^{k}_{X}J(\exp_{x^{m}}s\theta X)ds,\nnum\\\hspace{-0.8cm}&& 0<\theta<\theta^{*},\EN
which is equivalent to  
 \EQ \label{tay2}&&\hspace{-.5cm}J(\exp_{x^{m}}\theta X)=J(x^{m})+\theta (dJ(X))|_{x^{m}}+...+\frac{\theta^{k-1}}{(k-1)!}\times\nnum\\&&\hspace{-.5cm}(\overset{p}{\nabla}^{k-2}_{X}dJ)(X)|_{x^{m}}+\frac{\theta^{k}}{(k-1)!}\int^{1}_{0}(1-s)^{k-1}(\overset{p}{\nabla}^{k-1}_{X}dJ)(X)\nnum\\&&\hspace{-.5cm}(\exp_{x^{m}}s\theta X)ds,\quad 0<\theta<\theta^{*},\EN
 where $dJ:TM^{m}\rightarrow \mathds{R}$ is a differential form of $J$, $\theta^{*}$ is the upper existence limit for geodesics on $M^{m}$. Note that for compact manifolds $\theta^{*}=\infty$.

Hence, the expansion above along the geodesic curve $\Gamma(\theta,t)$ gives 
 \EQ &&\hspace{-.5cm}J(\Gamma(1,t))=J(\Gamma(0,t))+(\overset{p}{\nabla}_{X}J)(\Gamma(0,t))+\cdots+\nnum\\&&\hspace{-.5cm}\frac{1}{(k-1)!}(\overset{p}{\nabla}^{k-1}_{X}J)(\Gamma(0,t))+\frac{1}{(k-1)!}\int^{1}_{0}(1-s)^{k-1}\overset{p}{\nabla}^{k}_{X}\nnum\\&&\hspace{-.5cm}J(\Gamma(s,t))ds,\nnum\EN
where $X=\sum^{m}_{j=1}\sum^{n}_{i=1}a^{j}_{i}\sin(\omega^{j}_{i}t)\frac{\partial}{\partial x_{i}}$ which is due to the structure of $T_{x^{m}}M^{m}=T_{x_{1}}M\oplus\cdots\oplus T_{x_{m}}M$ and $\Gamma(0,t)=x^{m}$. 

Linear properties of $\nabla$ imply  that (see \cite{Lee2})
 \EQ \nabla_{\sum^{n}_{i=1}a^{j}_{i}\sin(\omega^{j}_{i}t)\frac{\partial}{\partial x_{i}}}J=\sum^{n}_{i=1}a^{j}_{i}\sin(\omega^{j}_{i}t)\nabla_{\frac{\partial}{\partial x_{i}}}J.\nnum\EN 
By (\ref{pr}) we have
\EQ \label{kos}\overset{p}{\nabla}_{\sum^{m}_{j=1}\sum^{n}_{i=1}a^{j}_{i}\sin(\omega^{j}_{i}t)\frac{\partial}{\partial x_{i}}}J=\sum^{m}_{j=1}\sum^{n}_{i=1}a^{j}_{i}\sin(\omega^{j}_{i}t)\nabla_{\frac{\partial}{\partial x_{i}}}J,\EN 
where $J$ in the right hand side of (\ref{kos}) is restricted to $M$.

 Iteratively we have
 \EQ &&\overset{p}{\nabla}^{k}_{\sum^{m}_{j=1}\sum^{n}_{i=1}a^{j}_{i}\sin(\omega^{j}_{i}t)\frac{\partial}{\partial x_{i}}}J=\nnum\\&&\sum^{m}_{j=1}\sum^{n}_{i=1}a^{j}_{i}\sin(\omega^{j}_{i}t)\nabla_{\frac{\partial}{\partial x_{i}}}\big(\overset{p}{\nabla}^{k-1}_{\sum^{m}_{j=1}\sum^{n}_{i=1}a_{i}\sin(\omega_{i}t)\frac{\partial}{\partial x_{i}}}J\big),\nnum\EN
where $\overset{p}{\nabla}^{k-1}_{\sum^{m}_{j=1}\sum^{n}_{i=1}a_{i}\sin(\omega_{i}t)\frac{\partial}{\partial x_{i}}}J$ is decomposed on $T_{x_{j}}M,\hspace{.2cm}j=1,\cdots,m$.

	We drop the notation $\hat{}$ for the state trajectory in (\ref{grg}). Hence, the synchronization algorithm (\ref{grg}) and the dynamical equations for the extremum seeking feedback loop are given in $x$ coordinates as follows:  
\EQ \label{koonkir}&&\dot{x}_{j}(t)=-\Big(\sum^{n}_{i=1}a^{j}_{i}\sin(\omega^{j}_{i}t)J(x^{m})\frac{\partial}{\partial x_{i}}+\nnum\\&&\sum^{m}_{e=1}\sum^{n}_{i,l=1}a^{j}_{i}a^{e}_{l}\sin(\omega^{j}_{i}t)\sin(\omega^{e}_{l}t)\nabla_{\frac{\partial}{\partial x_{l}}}J(x^{m})\frac{\partial}{\partial x_{i}}+\cdots+\nnum\\&&\frac{1}{(k-1)!}\sum^{m}_{e=1}\sum^{n}_{i,l=1}a^{j}_{i}a^{e}_{l}\sin(\omega^{j}_{i}t)\sin(\omega^{e}_{l}t)\times\nnum\\&& \nabla_{\frac{\partial}{\partial x_{i}}}\big(\overset{p}{\nabla}^{k-2}_{\sum^{m}_{e=1}\sum^{n}_{i=1}a^{e}_{i}\sin(\omega^{e}_{i}t)\frac{\partial}{\partial x_{i}}}J\big)(x^{m})\frac{\partial}{\partial x_{i}}+\nnum\\&&\frac{1}{(k-1)!}\sum^{n}_{i=1}a^{j}_{i}\sin(\omega^{j}_{i}t)\Big(\nnum\\&&\int^{1}_{0}(1-s)^{k-1}\overset{p}{\nabla}^{k}_{\sum^{m}_{e=1}\sum^{n}_{i=1}a^{e}_{i}\sin(\omega^{e}_{i}t)\frac{\partial}{\partial x_{i}}}\nnum\\&&J(\cdots,\exp_{x_{j}}s \sum^{n}_{i=1}a^{j}_{i}\sin(\omega^{j}_{i}t)\frac{\partial}{\partial x_{i}},\cdots)ds\Big)\frac{\partial}{\partial x_{i}}\Big).\EN 
	The vector field above is a time varying vector field on $M$ for each agent $j$. Since the perturbations appear in the form of  sinusoids then by Assumption \ref{as} the resulted vector field is periodic with respect to time.
	By employing Assumption \ref{as}, the averaged vector field for agent $j$ is given by
	\EQ \label{av}&&\frac{1}{2}\sum^{n}_{i=1}a^{j^{2}}_{i}\nabla_{\frac{\partial}{\partial x_{l}}}J(x^{m})\frac{\partial}{\partial x_{i}}+\cdots+\nnum\\&&\frac{1}{T}\int^{T}_{0}\Big[\frac{1}{(k-1)!}\sum^{m}_{e=1}\sum^{n}_{i,l=1}a^{j}_{i}a^{e}_{l}\sin(\omega^{j}_{i}t)\sin(\omega^{e}_{l}t)\times\nnum\\&& \nabla_{\frac{\partial}{\partial x_{i}}}\big(\overset{p}{\nabla}^{k-2}_{\sum^{m}_{e=1}\sum^{n}_{i=1}a^{e}_{i}\sin(\omega^{e}_{i}t)\frac{\partial}{\partial x_{i}}}J\big)(x^{m})\frac{\partial}{\partial x_{i}}+\nnum\\&&\frac{1}{(k-1)!}\sum^{n}_{i=1}a^{j}_{i}\sin(\omega^{j}_{i}t)\Big(\nnum\\&&\int^{1}_{0}(1-s)^{k-1}\overset{p}{\nabla}^{k}_{\sum^{m}_{e=1}\sum^{n}_{i=1}a^{e}_{i}\sin(\omega^{e}_{i}t)\frac{\partial}{\partial x_{i}}}\nnum\\&&J(\cdots,\exp_{x_{j}}s \sum^{n}_{i=1}a^{j}_{i}\sin(\omega^{j}_{i}t)\frac{\partial}{\partial x_{i}},\cdots)ds\Big)\frac{\partial}{\partial x_{i}}\Big]dt.\nnum\\\EN
	We observe that each agent can select its geodesic dither amplitudes $a^{j}_{i},\hspace{.2cm}j=1,\cdots,m, i=1,\cdots,n$ such that $||\sum^{n}_{i=1}a^{j}_{i}\sin(\omega^{j}_{i}t)\frac{\partial}{\partial x_{i}}||<i(x_{j})$, where $i(x_{j})$ is the objectivity radius at $x_{j}\in M$, see Definition \ref{inj}. It is guaranteed that $i(x)>0$ for all $x\in M$, see \cite{Klin}. Hence, the set $\overline{\exp_{x_{j}}B_{i(x_{j})}(0)}=\exp_{x_{j}}\overline{B_{i(x_{j})}(0)}$ is compact in the topology of $M$.  This results in the compactness of $\exp_{x_{1}}\overline{B_{i(x_{1})}(0)}\times \cdots\times \exp_{x_{m}}\overline{B_{i(x_{m})}(0)}$
in the product topology of $M^{m}$. By the choice of dither frequencies and Assumption \ref{as} we have $\frac{1}{T}\int^{T}_{0}\sin(\omega^{j}_{i}t)\sin(\omega^{k}_{l}t)\sin(\omega^{q}_{r}t)dt=0,\hspace{.2cm}j,k,q\in 1,\cdots,m, i,l,r\in 1,\cdots,n$. Together with the smoothness of $J$ and compactness of $\exp_{x_{1}}\overline{B_{i(x_{1})}(0)}\times \cdots\times \exp_{x_{m}}\overline{B_{i(x_{m})}(0)}$, this implies that the only significant term after the integration in (\ref{av}) is of order $O(\max_{i\in\{ 1,\cdots,n\},j\in \{1,\cdots,m\}}|a^{j}_{i}|^{4})$. 
	Consequently, the averaged vector field for each agent is in the form of a perturbation in the statement of the lemma and the proof is complete. 
	\end{proof}
	\begin{remark}
	The order of the perturbation vector field constructed above is not uniform with respect to $M^{m}$, since the injectivity radius may vary on Riemannian manifolds. However, in the case that $i(M)$ is bounded from below the perturbation term in (\ref{av}) can be uniformly bounded. This specially holds for compact manifolds since the injectivity radius of compact manifolds are bounded from below, see Lemma \ref{kl}. 
	\end{remark}
	\subsection{Stability of the gradient system on $G^{m}/G_{c}$}
	As stated in the previous section, the averaged vector field of each agent is in the perturbation form represented in (\ref{av}). The results of Lemma \ref{l4} also hold for $G^{m}$ instead of $M^{m}$ since by definition Lie groups are smooth manifolds. We define the gradient system of the synchronization problem on $M^{m}$as follows.
		\begin{definition}
		\label{gs}
		For the synchronization extremum seeking algorithm (\ref{grg}), the \textit{gradient vector field} on $M^{m}$ is given by $\sum^{m}_{j=1}\sum^{n}_{i=1}\frac{1}{2}a^{j^{2}}_{i}\nabla_{\frac{\partial}{\partial x_{i}}}J(x^{m})\frac{\partial}{\partial x_{i}}\in \mathfrak{X}(M^{m})$. 
		\end{definition}
		Note that $\sum^{n}_{i=1}\frac{1}{2}a^{j^{2}}_{i}\nabla_{\frac{\partial}{\partial x_{i}}}J(x^{m})\frac{\partial}{\partial x_{i}}\in \mathfrak{X}(M)$ for which we consider their unique extensions presented in (\ref{le}) on $M^{m}$. As stated before the synchronization cost on $G^{m}$ has a set of minima denoted by $G_{c}\subset G^{m}$.   In order to analyze the synchronization problem on $G^{m}$ we modify the extremum seeking vector field (\ref{grg}) to be applicable on Lie groups.
		The extremum seeking algorithm for the synchronization on $G^{m}$ is given as
		\EQ \label{grg1}&&\hspace{-.5cm}\dot{g}_{j}=-\sum^{n}_{i=1}a^{j}_{i}\sin(\omega_{i}t)\times \nnum\\&&J(\cdots,g_{j}\star\exp(\sum^{n}_{i=1}a^{j}_{i}\sin(\omega^{j}_{i}t)\frac{\partial}{\partial g_{i}}),\cdots)g_{j}\frac{\partial}{\partial g_{i}},\hspace{.2cm}\nnum\\&&j=1,\cdots,m,\EN
		where $\exp$ is the exponential map on Lie groups defined in (\ref{kirkir1}) and $\frac{\partial}{\partial g_{i}}$ are the base elements of $\mathcal{L}$. In this case we employ the left invariant vector field, denoted by $g_{j}\frac{\partial}{\partial g_{i}}$, induced by $\frac{\partial}{\partial g_{i}}$ on $\mathcal{L}$ given by $T_{e}g(\frac{\partial}{\partial g_{i}})$. One may show that $T_{e}(g_{1}\star g_{2})(\frac{\partial}{\partial g_{i}})=T_{g_{2}}g_{1}\circ T_{e}g_{2}(\frac{\partial}{\partial g_{i}})$ which shows that $T_{e}g(\frac{\partial}{\partial g_{i}})$ is left invariant.
Note that the $\exp$ curve is not necessarily a geodesic on $G$. It is shown that in the case which $G$ admits a bi-invariant Riemannian metric the exponential curves through $e$ are geodesics, \cite{pen}. In this case it is easy to show that $\gamma(t)=g\star \exp(tX), \hspace{.2cm}X\in \mathcal{L}$ is a geodesic through $g\in G$ since
			\EQ \nabla_{\dot{\gamma}(t)}\dot{\gamma}(t)&=&\nabla_{T_{g\star \exp(tX)}g\frac{d\exp(tX)}{dt}}T_{g\star \exp(tX)}g\frac{d\exp(tX)}{dt}\nnum\\&=&T_{g\star \exp(tX)}g\nabla_{\frac{d\exp(tX)}{dt}}\frac{d\exp(tX)}{dt}=0,\nnum\EN
			since $\exp(tX)$ is a geodesic and $\nabla_{\frac{d\exp(tX)}{dt}}\frac{d\exp(tX)}{dt}=0$. Note that $\nabla$ is the corresponding invariant connection with respect to the Cartan-Schouten (0) form on $G$, see \cite{pen}.  Hence, the analysis of (\ref{grg1}) is exactly the same as the analysis of (\ref{grg}) in Lemma \ref{l4}. However, a bi-invariant Riemannian metric may not exist for all Lie groups. As an example $SE(3)$ does not admit such a metric and consequently the  exponential map on $SE(3)$ is not a geodesic, see \cite{pen}. In the case that $G$ does not admit a bi-invariant metric,  we employ the Taylor expansion of smooth functions on $G$ and replace (\ref{tay}) by its version on Lie groups, given in \cite{knap}. The rest of the analysis remains unchanged where $\nabla_{X}J(g)$ is replaced by $XJ(g)=\lim _{t\rightarrow 0}\frac{J(g\star\exp(tX(e)))-J(g)}{t}$. 
			
			The stability of the extremum seeking algorithm (\ref{grg1}) is related to the stability of the gradient system in Definition \ref{gs}. As explained, the synchronization cost function has the set of minima at $G_{c}=\{(g,\cdots,g)\in G^{m}\}$ for all $g\in G$. First we shows that the gradient system of the synchronization system in Definition \ref{gs} gives a left invariant vector field on $G^{m}$. 
			\begin{lemma}
			\label{li}
			Consider the gradient system of the synchronization extremum seeking algorithm (\ref{grg1}) which is given in Definition \ref{gs} on $G^{m}$. Assume $J$ satisfies Assumption \ref{as} then the gradient system is  invariant with respect to $G_{c}$ on $G^{m}$. 
			\end{lemma}
			\begin{proof}
			 Applying the analysis of the proof of Lemma \ref{l4} implies that the gradient system of (\ref{grg1}) is in the following form
			\EQ \label{gvf1}\dot{g^{m}}=-\sum^{m}_{j=1}\sum^{n}_{i=1}\frac{1}{2}a^{j^{2}}_{i}\nabla_{g_{j}\frac{\partial}{\partial g_{i}}}J(g^{m})g_{j}\frac{\partial}{\partial g_{i}}.\EN
	
			The left invariance of the vector field $g_{j}\frac{\partial}{\partial g_{i}}$ is immediate. Also $\nabla_{g_{j}\frac{\partial}{\partial g_{i}}}J(g^{m})$ is $G_{c}$ invariant since for a general left invariant vector field $X\in \mathfrak{X}(G^{m})$ we have $\nabla_{X}J(g^{m})=(X(J))(g^{m})$. Together with the invariance of $J$  with respect to $G_{c}$, this implies that $X(J)(g^{m})=X(J)(g_{c}\overline{\star}g^{m})$. This is due to the fact that (see (\ref{tay}))
			\EQ\label{koonkoon} \nabla_{X}J(g_{c}\overline{\star}g^{m})&=&\lim_{t\rightarrow 0}\frac{J(g_{c}\overline{\star}g^{m}\overline{\star}\exp t X)-J(g_{c}\overline{\star}g^{m})}{t}\nnum\\&=&\lim_{t\rightarrow 0}\frac{J(g^{m}\overline{\star}\exp t X)-J(g^{m})}{t}\nnum\\
			&=&\nabla_{X}J(g^{m}).\EN
			Hence, the gradient vector field for each agent $j$ is $G_{c}$ invariant.
			\end{proof}
			\begin{remark}
			 The second equality in (\ref{koonkoon}) does not necessarily hold along geodesics since in general $\exp_{g_{c}\overline{\star}g^{m}}tX\ne g_{c}\overline{\star}\exp_{g^{m}}tX$. In this case we may not be able to use the invariance properties of $J$ with respect to $G_{c}$.
			\halmos\end{remark}
			\begin{theorem}
			\label{st}
			Consider the gradient system of the synchronization extremum seeking algorithm (\ref{grg1}) which is given in Definition \ref{gs} for all agents on $G^{m}$. Assume $J$ is positive,  $G_{c}$ invariant and $J(g)=0,\hspace{.2cm}g\in G_{c}$. Then, if the initial state $\pi(g^{m}(t_{0}))$ is suffciently close (in the quotient topology) to $\pi(G_{c})$, then the state trajectory of the induced gradient system on $G^{m}/G_{c}$ initiating from $\pi(g^{m}(t_{0}))$ asymptotically converges to $\pi(G_{c})\in G^{m}/G_{c}$. 
			\end{theorem}
			\begin{proof}
			By the results of Lemma \ref{li} the gradient vector field is $G_{c}$ invariant and consequently induces a vector field on $G^{m}/G_{c}$.   The  vector field in (\ref{koonkoon}) induces a vector field $\sum^{m}_{j=1}\sum^{n}_{i=1}\frac{1}{2}a^{j^{2}}_{i}T_{g^{m}}\pi\big(\nabla_{g_{j}\frac{\partial}{\partial g_{i}}}J(g^{m})g_{j}\frac{\partial}{\partial g_{i}}\big)\in T_{\pi(g^{m})}G^{m}/G_{c}$. 
			The cost function $J$ induces a smooth function $\hat{J}:G^{m}/G_{c}\rightarrow \mathds{R}$ via $J=\hat{J}\circ \pi$, where by using the the horizontal lift, we have  $\overset{H}{grad_{g}\hat{J}}=grad_{[g]}J,\hspace{.2cm}g\in G^{m}/G_{c}$, see \cite{abs}. Since $J$ is a single valued smooth function, based on (\ref{levi}), for $X\in \mathfrak{X}(G^{m})$ we have $\nabla_{X}J=X(J)=dJ(X)=g_{G^{m}}(grad J,X)$. The operator $T_{g^{m}}$ is linear and therefore the induced vector field denoted by $\hat{X}$ is evaluated at $\pi(g^{m})$ by
			\EQ \label{gvf}&&-\sum^{m}_{j=1}\sum^{n}_{i=1}\frac{1}{2}a^{j^{2}}_{i}\nabla_{g_{j}\frac{\partial}{\partial g_{i}}}J(g^{m})T_{g^{m}}\pi\big(g_{j}\frac{\partial}{\partial g_{i}}\big)=\nnum\\&&-\sum^{m}_{j=1}\sum^{n}_{i=1}\frac{1}{2}a^{j^{2}}_{i}g_{G^{m}}(grad_{g^{m}} J,g_{j}\frac{\partial}{\partial g_{i}})T_{g^{m}}\pi\big(g_{j}\frac{\partial}{\partial g_{i}}\big)=\nnum\\&&- \sum^{m}_{j=1}\sum^{n}_{i=1}\frac{1}{2}a^{j^{2}}_{i}g_{G^{m}/G_{c}}(grad_{\pi(g^{m})} \hat{J},T_{g^{m}}\pi( g_{j}\frac{\partial}{\partial g_{i}}))\times\nnum\\&& T_{g^{m}}\pi\big(g_{j}\frac{\partial}{\partial g_{i}}\big)=\nnum\\&&-\sum^{m}_{j=1}\sum^{n}_{i=1}\frac{1}{2}a^{j^{2}}_{i}\hat{\nabla}_{T_{g^{m}}\pi\big(g_{j}\frac{\partial}{\partial g_{i}}\big)}\hat{J}(\pi(g^{m}))T_{g^{m}}\pi\big(g_{j}\frac{\partial}{\partial g_{i}}\big),\nnum\\\EN
			
			where $\hat{\nabla}$ is the Levi-Civita connection of $G_{m}/G_{c}$.
			Since $J(g)=0, g\in G_{c}$ then $\hat{J}\circ \pi(g)=0,\hspace{.2cm}g\in G_{c}$. Note that for all $g_{1},g_{2}\in G_{c}$ we have $\pi(g_{1})=\pi(g_{2})$, hence $G_{c}$ maps to a single point $\pi(G_{c})$ in $G^{m}/G_{c}$. It is immediate that $\pi(G_{c})$  is a unique local minimum of $\hat{J}$ since for any $\hat{g}\ne \pi(G_{c})$ we have $\pi^{-1}(\hat{g})\cap G_{c}=\emptyset$. Otherwise there exists $g_{1}\in \pi^{-1}(\hat{g})$ such that $g_{1}=(g_{11},\cdots,g_{11})\in G_{c}$. Since $g_{1}\sim \pi^{-1}(\hat{g})$ then for each $g_{2}\in \pi^{-1}(\hat{g})$, there exists $g_{c}\in G_{c}$ such that $g_{2}=g_{c}\overline{\star} g_{1}$. This implies $\pi^{-1}(\hat{g})=G_{c}$ or $\hat{g}=\pi(G_{c})$ which is a contradiction.   We consider $\hat{J}$ as a candidate Lyapunov  function on $G^{m}/G_{c}$. The time variation of $\hat{J}$ along the induced gradient vector field (\ref{gvf}) is given as
				\EQ \dot{\hat{J}}&=&d\hat{J}(\hat{X})\nnum\\&=&-\sum^{m}_{j=1}\sum^{n}_{i=1}\frac{1}{2}a^{j^{2}}_{i}\hat{\nabla}_{T_{g^{m}}\pi\big(g_{j}\frac{\partial}{\partial g_{i}}\big)}\hat{J}(\pi(g^{m}))\times \nnum\\&&\left.d\hat{J}\Big(T_{g^{m}}\pi\big(g_{j}\frac{\partial}{\partial g_{i}}\big)\Big)\right|_{\pi(g^{m})}.\nnum\EN
				As is obvious $\left.d\hat{J}\Big(T_{g^{m}}\pi\big(g_{j}\frac{\partial}{\partial g_{i}}\big)\Big)\right|_{\pi(g^{m})}=\hat{\nabla}_{T_{g^{m}}\pi\big(g_{j}\frac{\partial}{\partial g_{i}}\big)}\hat{J}(\pi(g^{m}))$. Hence,
				\EQ \dot{\hat{J}}=-\sum^{m}_{j=1}\sum^{n}_{i=1}\frac{1}{2}a^{j^{2}}_{i}\hat{\nabla}^{2}_{T_{g^{m}}\pi\big(g_{j}\frac{\partial}{\partial g_{i}}\big)}\hat{J}(\pi(g^{m}))\leq 0.\nnum\EN
		Since $\pi(G_{c})$ is the unique local minimum of $\hat{J}$ then $\hat{\nabla}_{T_{g^{m}}\big(g_{j}\frac{\partial}{\partial g_{i}}\big)}\hat{J}(\pi(g^{m}))=0$ if and only if $g^{m}\in G_{c}$. This yields that $\dot{\hat{J}}$ locally vanishes only at $\pi(G_{c})\in G^{m}/G_{c}$.  By employing the Lyapunov stability results on manifolds, see \cite{Lewis}, $\pi(G_{c})$ is locally asymptotically stable on $G^{m}/G_{c}$ and the proof is complete. 
			\end{proof}
		One may show that asymptotic convergence of the state trajectory of the induced gradient system in the quotient manifold $G^{m}/G_{c}$ results in the asymptotic convergence in $G^{m}$. 
			Consider the curve $\gamma(t)\doteq\pi(\Phi_{X}(t,t_{0},g^{m}))$ on $G^{m}/G_{c}$, where $\Phi_{X}$ is the flow of $X$ on $G^{m}$, see (\ref{flow}). To show the convergence in $G^{m}$ we need to show $\gamma(t)=\hat{\gamma}(t)\doteq\Phi_{\hat{X}}(t,t_{0},\pi(g^{m}))$ in $G^{m}/G_{c}$, where $\Phi_{\hat{X}}$ is the  flow of $\hat{X}$ on $G^{m}/G_{c}$. To this end, it is sufficient to prove both of them are integral flows of the same vector field with the same initial conditions. Obviously both $\gamma$ and $\hat{\gamma}$ initiate from the same initial state $\pi(g^{m})\in G^{m}/G_{c}$ and $\Phi_{\hat{X}}(t,t_{0},\pi(g^{m}))$ is the solution of the vector field $\hat{X}$ on $G^{m}/G_{c}$.  The tangent vector field along $\gamma$ in $G^{m}/G_{c}$ is obtained by 
			\EQ \label{fuck}\dot{\gamma}(t)&=&T_{\Phi_{X}(t,t_{0},g^{m})}\pi X(\Phi_{X}(t,t_{0},g^{m}))\nnum\\&=&\hat{X}(\pi(\Phi_{X}(t,t_{0},g^{m}))),\EN
			where the second equality holds since $\hat{X}$ is the horizontal lift of $X$, see (\ref{hl}).  Equation (\ref{fuck}) shows that $\pi(\Phi_{X}(t,t_{0},g^{m}))$ is the solution of the vector field $\hat{X}$ in $G^{m}/G_{c}$ with initial conditions $\pi(g^{m})\in G^{m}/G_{c}$ and $\hat{X}(\pi(g^{m}))\in T_{\pi(g^{m})}G^{m}/G_{c}$. Hence, by the uniqueness of solutions for flows we have $\gamma(t)=\hat{\gamma}(t)$. As stated by Theorem \ref{st}, if $\pi(g^{m})$ is sufficeintly close to $\pi(G_{c})$, then $\Phi_{\hat{X}}(t,t_{0},\pi(g^{m}))\rightarrow \pi(G_{c})$. Hence, together with  continuity of $\pi$ in the quotient topology, we have $\Phi_{X}(t,t_{0},g^{m})\rightarrow G_{c}$.  This is summarized in the following proposition.  
			\newtheorem{proposition}{Proposition}
			\begin{proposition}
			Consider the initial state $g^{m}\in G^{m}$ such that $\Phi_{\hat{X}}(t,t_{0},\pi(g^{m}))\rightarrow \pi(G_{c})$, where $\hat{X}$ is the induced gradient vector field (\ref{gvf}) and $\Phi_{\hat{X}}$ is its flow as per (\ref{flow}). Then $\Phi_{X_{g}}(t,t_{0},g)\rightarrow G_{c}$, where $X_{g}$ is the gradient vector field (\ref{gvf1}). 
			\halmos\end{proposition}
					\subsection{Closeness of solutions on $G^{m}/G_{c}$}
					To analyze the behaviour of the extremum seeking algorithm (\ref{grg1}) on $G^{m}$ we need to study the closeness of solutions of perturbed vector fields on $G^{m}$. As stated by Theorem \ref{st} for sufficiently close initial state $\pi(g^{m})$ the state flow $\Phi_{\hat{X}}(t,t_{0},\pi(g^{m})$ converges to $\pi(G_{c})$. However, the original state trajectory $\Phi_{X}(t,t_{0},g^{m})$ converges to the invariant set $G_{c}$ which is not a single point. To obtain the closeness of solutions for state trajectories of (\ref{grg1}) and its corresponding gradient system in Definition \ref{gs} we study their projected trajectories on $G^{m}/G_{c}$. 
					\begin{lemma}
					\label{l5}
					Consider the synchronization extremum seeking algorithm (\ref{grg1}) on the connected Lie group $G^{m}$ such that $i(G^{m})$ is bounded from below. Then  the averaged vector field of the synchronization extremum seeking algorithm, $X_{a}$,  is $G_{c}$ invariant and there exists a continuous function $\rho:\mathds{R}\rightarrow \mathds{R},\hspace{.2cm}\rho(0)=0$, such that
					\EQ&& \limsup_{t\rightarrow \infty}d(\Phi_{X_{a}}(t,t_{0},g^{m}),G_{c})\leq\nnum\\&& \rho(O(\max_{i\in\{ 1,\cdots,n\},j\in \{1,\cdots,m\}}|a^{j}_{i}|^{4})).\nnum\EN
					\end{lemma}
					
					\begin{proof}
					As shown by Lemma \ref{l4} the averaged vector field is the perturbation of the gradient system defined in Definition \ref{gs}. Let us denote the time varying synchronization vector field in (\ref{grg1}) by $X(g,t)$. Since $J$ is $G_{c}$ invariant and $g\frac{\partial}{\partial g_{i}}$ are left invariant then it is immediate that for each $t$, $T_{g}g_{c}X(g,t)=X(g_{c}\overline{\star}g,t),\hspace{.2cm}g_{c}\in G_{c}$.  Since $X$ is $T$ periodic then $\frac{1}{T}\int^{T}_{0}X(g,\tau)d\tau$ is also $G_{c}$ invariant. By the results of Lemma \ref{l4} we have
					\EQ \hspace{-.5cm}X_{a}(g^{m})&=&-\sum^{m}_{j=1}\sum^{n}_{i=1}\frac{1}{2}a^{j^{2}}_{i}\nabla_{g_{j}\frac{\partial}{\partial g_{i}}}J(g^{m})g_{j}\frac{\partial}{\partial g_{i}}+\nnum\\&&\sum^{m}_{j=1}\sum^{n}_{i=1}O(\max_{i\in\{ 1,\cdots,n\},j\in \{1,\cdots,m\}}|a^{j}_{i}|^{4})g_{j}\frac{\partial}{\partial g_{i}}.\nnum\EN
					Hence, the induced vector field on $G^{m}/G_{c}$ is given by
					\EQ &&\hat{X}_{a}(\pi(g^{m}))=T_{g^{m}}\pi(X_{a}(g^{m}))=\nnum\\&&-T_{g^{m}}\pi\left(\sum^{m}_{j=1}\sum^{n}_{i=1}\frac{1}{2}a^{j^{2}}_{i}\nabla_{g_{j}\frac{\partial}{\partial g_{i}}}J(g^{m})g_{j}\frac{\partial}{\partial g_{i}}\right)+\nnum\\&&T_{g^{m}}\pi\left(\sum^{m}_{j=1}\sum^{n}_{i=1}O(\max_{i\in\{ 1,\cdots,n\},j\in \{1,\cdots,m\}}|a^{j}_{i}|^{4})g_{j}\frac{\partial}{\partial g_{i}}\right).\nnum\EN
					As shown by Theorem \ref{st}, the induced vector field of the gradient system is locally asymptotic stable around $\pi(G_{c})$. Hence, $\hat{X}$ is a perturbation of an asymptotic stable vector field on $G^{m}/G_{c}$. By employing the results of \cite{Taringoo100}, there exists a continuous function $\rho:\mathds{R}\rightarrow \mathds{R}, \hspace{.2cm}\rho(0)=0$, such that 
					\EQ \label{koonkash}&&\limsup_{t\rightarrow \infty}d\left(\Phi_{\hat{X}_{a}}(t,t_{0},\pi(g^{m})),\pi(G_{c})\right)\leq\nnum\\&& \rho(O(\max_{i\in\{ 1,\cdots,n\},j\in \{1,\cdots,m\}}|a^{j}_{i}|^{4})). \EN
					Connectedness of $G^{m}/G_{c}$ implies that there exists a piecewise smooth $\hat{\gamma}:[0,1]\rightarrow G^{m}/G_{c}$ such that $\hat{\gamma}(1)=\pi(G_{c})$ and $\hat{\gamma}(0)=\Phi_{\hat{X}_{a}}(t,t_{0},\pi(g^{m}))$. Results of \cite{kob}, Proposition II,3.1 yields the existence of  the unique horizontal lift of $\hat{\gamma}(\cdot)$ denoted by $\gamma(\cdot)\in G^{m}$, such that $T_{\gamma(t)}\pi\dot{\gamma}(t)=\dot{\hat{\gamma}}(t)$ and $\pi(\gamma(t))=\hat{\gamma}(t)$. Since $\gamma(\cdot)$ is a horizontal of $\hat{\gamma}(\cdot)$ then $\ell(\gamma)=\int^{1}_{0}g^{\frac{1}{2}}_{G^{m}}(\dot{\gamma}(\tau),\dot{\gamma}(\tau))d\tau=\int^{1}_{0}g^{\frac{1}{2}}_{G^{m}/G_{c}}(\dot{\hat{\gamma}}(\tau),\dot{\hat{\gamma}}(\tau))d\tau=\ell(\hat{\gamma})$. Therefore,
					\EQ d(\Phi_{X_{a}}(t,t_{0},g^{m}),G_{c})\leq d(\Phi_{X_{a}}(t,t_{0},g^{m}),\gamma(1))= \ell(\gamma),\nnum \EN
					where $d(\Phi_{X}(t,t_{0},g^{m}),G_{c})=\inf_{g_{c}\in G_{c}}d(\Phi_{X}(t,t_{0},g^{m}),g_{c})$ and $\gamma(1)\in G_{c}$.
					By the Riemannian structure of $G^{m}/G_{c}$ and continuity of $\rho$, select $a^{j}_{i}$ sufficiently small such that $\rho(O(\max_{i\in\{ 1,\cdots,n\},j\in \{1,\cdots,m\}}|a^{j}_{i}|^{4}))<i(\pi(G_{c}))$, where $i(G_{c})$ is the injectivity radius at $\pi(G_{c})$ in $G^{m}/G_{c}$. One may choose $\hat{\gamma}$ as the radial geodesic in a normal neighbourhood of $\pi(G_{c})$, see \cite{Lee3}. This implies that $d(\Phi_{\hat{X}_{a}}(t,t_{0},\pi(g^{m})),\pi(G_{c}))=\ell(\hat{\gamma})$. Hence, by (\ref{koonkash})
					\EQ && \limsup_{t\rightarrow \infty}d(\Phi_{X_{a}}(t,t_{0},g^{m}),G_{c})\leq\nnum\\&& \rho(O(\max_{i\in\{ 1,\cdots,n\},j\in \{1,\cdots,m\}}|a^{j}_{i}|^{4})),\nnum\EN
		which completes the proof.			
					\end{proof}
					The next theorem is the main result of this paper which gives closeness of solutions for state trajectories of dynamical systems on $G^{m}$. 
					\begin{theorem}
	\label{t1}
  Consider the synchronization extremum seeking system given in (\ref{grg1}) on  $G^{m}$. 
  Subject to Assumption \ref{as}, for any neighbourhood $U_{\pi(G_{c})}\subset G^{m}/G_{c}$ of $\pi(G_{c})$ on $G^{m}/G_{c}$, there exist a neighborhood $\hat{U}_{\pi(G_{c})}\subset G^{m}/G_{c}$ of $\pi(G_{c})$ such that for any $g^{m}_{0}\in G^{m},\hspace{.2cm}\pi(g^{m}_{0})\in \hat{U}_{\pi(G_{c})}$ there exist sufficiently small parameters $a^{j}_{i},\hspace{.2cm}i=1,\cdots,n, j=1,\cdots,m$ and sufficiently large frequency $\omega$, where the projected state trajectory of the closed loop system in (\ref{grg1}) on $G^{m}/G_{c}$ ultimately enters and remains in $U_{\pi(G_{c})}$.
  \end{theorem}
  
  \begin{proof}
	We analyze the closeness of solutions between state trajectories of  (\ref{grg1}) and the state trajectory of the gradient system on the quotient manifold $G^{m}/G_{c}$. 
	As stated in Assumption \ref{as}, the geodesic dithers frequencies are $\omega^{j}_{i}=\omega\bar{\omega}^{j}_{i},\hspace{.2cm}j\in(1,\cdots,m),i\in(1,\cdots,n)$. In the time scale $\tau=\omega t$ we have 
	\EQ \label{vagintala}&&\hspace{-.5cm}\frac{d g_{j}}{d\tau}=-\frac{1}{\omega}\sum^{n}_{i=1}a^{j}_{i}\sin(\bar{\omega}_{i}\tau)\times \nnum\\&&J(\cdots,g_{j}\star\exp(\sum^{n}_{i=1}a^{j}_{i}\sin(\bar{\omega}^{j}_{i}\tau)\frac{\partial}{\partial g_{i}}),\cdots)g_{j}\frac{\partial}{\partial g_{i}}\nnum\\&&\doteq \frac{1}{\omega}X(\tau,g^{m})\in T_{g^{m}}G^{m},\hspace{.2cm}j=1,\cdots,m.\EN 
	By the results of Lemma \ref{l4} the averaged dynamical system on $G^{m}$ is given by
	\EQ\label{vagin}\frac{dg^{m}}{d\tau}&=&-\frac{1}{\omega}\sum^{m}_{j=1}\sum^{n}_{i=1}\frac{1}{2}a^{j^{2}}_{i}\nabla_{g_{j}\frac{\partial}{\partial g_{i}}}J(g^{m})g_{j}\frac{\partial}{\partial g_{i}}+\nnum\\&&\frac{1}{\omega}\sum^{m}_{j=1}\sum^{n}_{i=1}O(\max_{i\in\{ 1,\cdots,n\},j\in \{1,\cdots,m\}}|a^{j}_{i}|^{4}))g_{j}\frac{\partial}{\partial g_{i}}\nnum\\&\doteq& \frac{1}{\omega}X_{a}(g^{m})\in T_{g^{m}}G^{m},\nnum\\\EN
	which is in a form of a perturbation of the gradient vector field $-\frac{1}{\omega}\sum^{m}_{j=1}\sum^{n}_{i=1}\frac{1}{2}a^{j^{2}}_{i}\nabla_{g_{j}\frac{\partial}{\partial g_{i}}}J(g^{m})g_{j}\frac{\partial}{\partial g_{i}}\doteq \frac{1}{\omega}X_{g}(g^{m})$ on $G^{m}$. 
	As stated in the proof of Lemma \ref{l5}, the synchronization extremum seeking system (\ref{vagintala}) is left invariant with respect to $G_{c}$ and consequently induces time varying vector field $\hat{X}$, time invariant averaged vector field $\hat{X}_{a}$ and the induced gradient vector field $\hat{X}_{g}$ on $G^{m}/G_{c}$. Hence, we analyze closeness of solutions among the state trajectories of $\frac{1}{\omega}\hat{X},\frac{1}{\omega}\hat{X}_{a}$ and $\frac{1}{\omega}\hat{X}_{g}$ on $G^{m}/G_{c}$.

	 Consider the periodic vector field $Z(t,x)\doteq\int^{t}_{0}(\hat{X}_{a}(g)-\hat{X}(g,\tau))d\tau,\hspace{.2cm}g\in G^{m}/G_{c},\lambda\in\mathds{R}_{\geq 0}$,
 where $Z(t,g)=Z(t+T,g)$. Now consider a composition of flows on $G^{m}/G_{c}$ given by
\EQ z(\tau)&=&\Phi^{(1,0)}_{\frac{1}{\omega} Z}\circ \Phi_{\frac{1}{\omega} \hat{X}}(\tau,\tau_{0},g_{0})\nnum\\&\doteq& \Phi_{\frac{1}{\omega} Z}(1,0,\Phi_{\frac{1}{\omega} \hat{X}}(\tau,\tau_{0},g_{0})).\nnum\EN
The tangent vector of $z$ is computed by
\EQ \label{rrr}\dot{z}(\tau)&=&T_{\Phi_{\frac{1}{\omega} \hat{X}}(\tau,\tau_{0},g_{0})}\Phi_{\frac{1}{\omega} Z}^{(1,0)}\Big(\frac{1}{\omega} \hat{X}(\Phi_{\frac{1}{\omega} \hat{X}}(\tau,\tau_{0},g_{0}),\tau)\Big)\nnum\\&&+\frac{\partial}{\partial \tau}\big(\Phi_{\frac{1}{\omega} Z}^{(1,0)}\circ \Phi_{\frac{1}{\omega} \hat{X}}(\tau,\tau_{0},g_{0})\big)\nnum\\&=&(\Phi^{-1})_{\frac{1}{\omega} Z}^{(1,0)^{*}}\Big(\frac{1}{\omega} \hat{X}(\cdot,\tau)\Big)(z(\tau))\nnum\\&&+\frac{1}{\omega}\int^{1}_{0}(\Phi^{-1})^{(1,s)^{*}}_{\frac{1}{\omega} Z}\big(\hat{X}_{a}(\cdot)-\hat{X}(\cdot,\tau)\big)ds\circ z(\tau),\nnum\\\EN 
where $(\Phi^{-1})^{(1,s)^{*}}_{\frac{1}{\omega} Z}$ is the pullback of the state flow $\Phi^{-1}_{\frac{1}{\omega} Z}$ and $\epsilon=\frac{1}{\omega}$. See \cite{Lewis, Agra} for the definition of pullbacks along diffeomorphisms.
Equivalently, in a compact form, we have
\EQ \label{kk1}\dot{z}(\tau)&=&\frac{1}{\omega}\Big[(\Phi^{-1})_{\frac{1}{\omega} Z}^{(1,0)^{*}} \hat{X}\nnum\\&&+\int^{1}_{0}(\Phi^{-1})^{(1,s)^{*}}_{\frac{1}{\omega} Z}\big(\hat{X}_{a}-\hat{X}\big)ds\Big]\circ z(\tau)\nnum\\&\doteq&\frac{1}{\omega} H(\frac{1}{\omega},\tau,z(\tau)). \EN
One can see that $H(0,\tau,x)=\hat{f}(x)$ where by the construction above, $H$ is smooth with respect to $\frac{1}{\omega}$. By applying the Taylor expansion with remainder we have
\EQ H(\frac{1}{\omega},\tau,x)=\hat{X}_{a}(g)+\frac{1}{\omega} h(g,\zeta,\tau),\nnum\EN 
where $h(g,\zeta,\tau)=\frac{\partial }{\partial \frac{1}{\omega}}H(\frac{1}{\omega},\tau,g)|_{\frac{1}{\omega}=\zeta}$ and $\zeta\in[0,\frac{1}{\omega}]$.
We note that $H(\frac{1}{\omega},\tau,g)$ is periodic with respect to $\tau$ since $\hat{X}(g,\tau)$ and $Z(\tau,g)$ are both T-periodic. Hence, $h(g,\zeta,\tau)$ is a T-periodic vector field on $M$. 

The metric triangle inequality on $G^{m}/G_{c}$ implies
\EQ \label{kirkir}&&d(\Phi_{\frac{1}{\omega} \hat{X}}(\tau,\tau_{0},g_{0}), \Phi_{\frac{1}{\omega} \hat{X}_{a}}(\tau,\tau_{0},g_{0}))\leq\nnum\\&& d(\Phi_{\frac{1}{\omega} \hat{X}}(\tau,\tau_{0},g_{0}),\Phi_{\frac{1}{\omega} Z}^{(1,0)}\circ \Phi_{\frac{1}{\omega} \hat{X}}(\tau,\tau_{0},g_{0}))\nnum\\&&+ d(\Phi_{\frac{1}{\omega} Z}^{(1,0)}\circ \Phi_{\frac{1}{\omega} \hat{X}}(\tau,\tau_{0},g_{0}),\Phi_{\frac{1}{\omega} \hat{X}_{a}}(\tau,\tau_{0},g_{0}))\leq \nnum\\&&
d(\Phi_{\frac{1}{\omega} \hat{X}}(\tau,\tau_{0},g_{0}),\Phi_{\frac{1}{\omega} Z}^{(1,0)}\circ \Phi_{\frac{1}{\omega} \hat{X}}(\tau,\tau_{0},g_{0}))+\nnum\\&& d(\Phi_{\frac{1}{\omega} Z}^{(1,0)}\circ \Phi_{\frac{1}{\omega} \hat{X}}(\tau,\tau_{0},x_{0}),\pi(G_{c}))+\nnum\\&&d(\Phi_{\frac{1}{\omega} \hat{X}_{a}}(\tau,\tau_{0},g_{0}),\pi(G_{c})).\EN

Based on (\ref{kirkir}), We analyze the closeness of solutions for the following dynamics on $G^{m}/G_{c}$.
\EQ \label{kirekhar}&&\hspace{-.5cm}\frac{dg}{dt}=\hat{X}_{g}\left(g\right),\nnum\\&&\hspace{-.5cm}\frac{dg}{dt}=\hat{X}_{g}\left(g\right)+T_{g^{m}}\pi\left(\sum^{m}_{j=1}\sum^{n}_{i=1}O((\max_{j\in1,\cdots,m, i\in1,\cdots,n}a^{j}_{i})^{4})\times\right. \nnum\\&&\hspace{-0cm}\left. g_{j}\frac{\partial}{\partial g_{i}}\right),\nnum\\&&\hspace{-.5cm}\frac{dg}{dt}= \hat{X}_{g}\left(g\right)+T_{g^{m}}\pi\left(\sum^{m}_{j=1}\sum^{n}_{i=1}O((\max_{j\in1,\cdots,m, i\in1,\cdots,n}a^{j}_{i})^{4})\times \right.\nnum\\&&\left. g_{j}\frac{\partial}{\partial g_{i}}\right)+\frac{1}{\omega} h(g,\zeta,t),\EN
where $g(t_{0})=g_{0}$ and $g=\pi(g^{m})\in G^{m}/G_{c}$.

The variation of the induced cost function $\hat{J}$ along $\hat{X}_{a}=\hat{X}_{g}+T_{g^{m}}\pi\left(\sum^{m}_{j=1}\sum^{n}_{i=1}O((\max_{j\in1,\cdots,m, i\in1,\cdots,n}a^{j}_{i})^{4}) g_{j}\frac{\partial}{\partial g_{i}}\right)$ is given by\\
$\mathcal{L}_{\hat{X}_{g}}\hat{J}+\mathcal{L}_{T_{g^{m}}\pi\left(\sum^{m}_{j=1}\sum^{n}_{i=1}O((\max_{j\in1,\cdots,m, i\in1,\cdots,n}a^{j}_{i})^{4}) g_{j}\frac{\partial}{\partial g_{i}}\right)}\hat{J}$, where by the results of Theorem \ref{st} we have $\mathcal{L}_{\hat{X}_{g}}\hat{J}\leq 0$.
Without loss of generality, assume positive definiteness and negative semi definiteness of $\hat{J}$ and  $\mathcal{L}_{\hat{X}_{g}}\hat{J}$ are both obtained on the same neighbourhood on $G^{m}/G_{c}$. Otherwise we apply the intersection of the corresponding neighborhoods to perform the analysis above.
The sublevel set $\mathcal{N}_{b}$ of the cost  function $\hat{J}:G^{m}/G_{c}\rightarrow\mathds{R}_{\geq 0}$ on $G^{m}/G_{c}$ is defined by $\mathcal{N}_{b}\doteq\{g\in G^{m}/G_{c}, \hspace{.2cm} \hat{J}(g)\leq b\}$. By $\mathcal{N}_{b}(g^{*})$ we denote a connected sublevel set of $G^{m}/G_{c}$ containing $g^{*}\in G^{m}/G_{c}$.

By Lemma 6.12 in \cite{Lewis}, there exists a compact subslevel set $\mathcal{N}_{b}(\pi(G_{c}))\subset U_{\pi(G_{c})}$, such that $\mathcal{N}_{b}(\pi(G_{c}))$ is compact.
Consider a neighborhood $W_{\pi(G_{c})}\subset \mathcal{N}_{b}(\pi(G_{c}))\subset U_{\pi(G_{c})}$. 
The set $\mathcal{N}_{b}(\pi(G_{c}))-W_{\pi(G_{c})}=\mathcal{N}_{b}(\pi(G_{c}))\bigcap W^{c}_{\pi(G_{c})}$ is compact since $W^{c}_{\pi(G_{c})}$ is closed and $\mathcal{N}_{b}(\pi(G_{c}))\bigcap W^{c}_{\pi(G_{c})}\subset \mathcal{N}_{b}(\pi(G_{c}))$ is a closed subset of the compact set $\mathcal{N}_{b}(\pi(G_{c}))$, which is consequently compact.  

Compactness of $\mathcal{N}_{b}(\pi(G_{c}))-W_{\pi(G_{c})}$ and continuity of the perturbed vector field $T_{g^{m}}\pi\left(\sum^{m}_{j=1}\sum^{n}_{i=1}O((\max_{j\in1,\cdots,m, i\in1,\cdots,n}a^{j}_{i})^{4}) g_{j}\frac{\partial}{\partial g_{i}}\right)$ on $G^{m}/G_{c}$ together imply that  by selecting $a^{j}_{i},\hspace{.2cm}j=1,\cdots,m, i=1,\cdots,n$ sufficiently small we have $\mathcal{L}_{\hat{X}_{a}}\hat{J}<0$ on $\mathcal{N}_{b}(\pi(G_{c}))-W_{\pi(G_{c})}$. This implies that the state trajectory $g(\cdot)$ initiating inside $\mathcal{N}_{b}(\pi(G_{c}))$ remains in $\mathcal{N}_{b}(\pi(G_{c}))$. 


The variation of $\hat{J}$ along $ \hat{X}_{a}\left(g\right)+\frac{1}{\omega} h(g,\zeta,t)$ is given by 

\EQ\label{soori}&& \mathcal{L}_{\hat{X}_{a}+\frac{1}{\omega} h(g,\zeta,\omega)}\hat{J}=\mathcal{L}_{\hat{X}_{a}}\hat{J}+\frac{1}{\omega} \mathcal{L}_{h(g,\zeta,t)}\hat{J}=\mathcal{L}_{\hat{X}_{g}}\hat{J}+\nnum\\&&\mathcal{L}_{T_{g^{m}}\pi\left(\sum^{m}_{j=1}\sum^{n}_{i=1}O((\max_{j\in1,\cdots,m, i\in1,\cdots,n}a^{j}_{i})^{4}) g_{j}\frac{\partial}{\partial g_{i}}\right)}\hat{J}\nnum\\&&+\frac{1}{\omega} \mathcal{L}_{h(g,\zeta,t)}\hat{J}.\EN
The same argument applies  to the variation of $\hat{J}$ along $\hat{X}_{a}\left(g\right)+\frac{1}{\omega} h(g,\zeta,t)$ and  for sufficiently small $a^{j}_{i}$ and sufficiently large $\omega$ the state trajectory of $\frac{dg}{dt}= \hat{X}_{a}\left(g\right)+\frac{1}{\omega} h(g,\zeta,t)$ remains bounded in $\mathcal{N}_{b}(\pi(G_{c}))$.

Denote the uniform normal neighborhood of $\pi(G_{c})\in G^{m}/G_{c}$ with respect to $U_{\pi(G_{c})}$  by $U^{n}_{\pi(G_{c})}$ (its existence is guaranteed by Lemma 5.12 in \cite{Lee3}). Consider a  geodesic ball of radius $\delta$ where $U^{n}_{\pi(G_{c})}\subset \exp_{\pi(G_{c})}(B_{\delta}(0)) $.   By definition, $\exp_{\pi(G_{c})}(B_{\delta}(0))$ is an open set containing $\pi(G_{c})$ in the topology of $G^{m}/G_{c}$. Therefore  one can shrink $b$ to $\acute{b}, 0<\acute{b}\leq b,$  such that  $\mathcal{N}_{\acute{b}}(\pi(G_{c}))\subset \exp_{\pi(G_{c})}(B_{\delta}(0))$. Hence, we can select the set of initial state such as $\Phi_{\hat{X}_{a}+\frac{1}{\omega}h}(\cdot,t_{0},g_{0})$ stays in a normal neighborhood of $\pi(G_{c})$. Hence, without loss of generality we assume $\mathcal{N}_{b}(\pi(G_{c}))\subset \exp_{\pi(G_{c})}(B_{\delta}(0))$.

Therefore, by employing the results of \cite{Taringoo100}, there exist  a neighborhood $U^{1}_{\pi(G_{c})}\subset int(\mathcal{N}_{b}(\pi(G_{c}))$ and a continuous function $\rho$, such that 
 \EQ \label{fuck1}&&\limsup_{t\rightarrow \infty} d(\Phi_{\hat{X}_{a}+\frac{1}{\omega}h}(t,t_{0},g_{0}),\pi(G_{c}))\leq\nnum\\&&\rho(||T_{g^{m}}\pi\left(\sum^{m}_{j=1}\sum^{n}_{i=1}O((\max_{j\in1,\cdots,m, i\in1,\cdots,n}a^{j}_{i})^{4}) g_{j}\frac{\partial}{\partial g_{i}}\right)+\nnum\\&&\frac{1}{\omega}h(g,\zeta,t)||_{g}),\hspace{.2cm} g_{0}\in U^{1}_{\pi(G_{c})},\nnum\\\EN
 where $\rho$ is a continuous function which crosses the origin. Note that (\ref{fuck1}) does not guarantee the convergence of the perturbed state trajectory to $\pi(G_{c})$. However, it gives a local closeness of solutions in terms of the Riemannian distance function $d$ to $\pi(G_{c})$ after elapsing enough time. 
 
 By employing the triangle inequality we have 
\EQ \label{kirkir2}&&d(\Phi_{\hat{X}_{a}+\frac{1}{\omega}h}(t,t_{0},g_{0}),\Phi_{\hat{X}_{a}}(t,t_{0},g_{0}))\leq\nnum\\&& d(\Phi_{ \hat{X}_{a}+\frac{1}{\omega}h}(t,t_{0},g_{0}),\pi(G_{c}))+d(\pi(G_{c}),\Phi_{\hat{X}_{a}}(t,t_{0},g_{0})),\nnum\\\EN
 where in  (\ref{kirkir2}), $d(\pi(G_{c}),\Phi_{\hat{X}_{a}}(t,t_{0},g_{0}))$ is ultimately bounded by Lemma \ref{l5} and  $d(\Phi_{ \hat{X}_{a}+\frac{1}{\omega}h}(t,t_{0},g_{0}),\pi(G_{c}))$ can be chosen arbitrarily small by (\ref{fuck1}).  
	In order to show the closeness of trajectories $\Phi_{ \hat{X}}(t,t_{0},g_{0})$ and $\Phi_{\hat{X}_{a}+\frac{1}{\omega}h}(t,t_{0},g_{0})$ in terms of  $d\left(\Phi_{ \hat{X}}(t,t_{0},g_{0}),\Phi_{\hat{X}_{a}+\frac{1}{\omega}h}(t,t_{0},g_{0})\right)$, we switch back to the time scale $\tau$. 
To this end, we prove $d(\Phi_{\frac{1}{\omega} \hat{X}}(\tau,\tau_{0},g_{0}),\Phi_{\frac{1}{\omega} Z}^{(1,0)}\circ \Phi_{\frac{1}{\omega} \hat{X}}(\tau,\tau_{0},g_{0}))=O(\frac{1}{\omega})$. Note that $\Phi_{\frac{1}{\omega}\hat{X}_{a}+\frac{1}{\omega^{2}}h}(\tau,\tau_{0},g_{0}))=\Phi_{\frac{1}{\omega} Z}^{(1,0)}\circ \Phi_{\frac{1}{\omega} \hat{X}}(\tau,\tau_{0},g_{0})$.
First we show that $\Phi_{\frac{1}{\omega} \hat{X}}(\cdot,\tau_{0},g_{0})$ remains in a compact subset of $G^{m}/G_{c}$ provided $g_{0}\in int(\mathcal{N}_{b}(\pi(G_{c})))$. As demonstrated by (\ref{soori}) by selecting $a^{j}_{i}$ sufficiently small and $\omega$ sufficiently large, there exists $W_{\pi(G_{c})}$ such that for $g_{0}\in W_{\pi(G_{c})}$ the state trajectory $\Phi_{ \hat{X}_{a}+\frac{1}{\omega}h}(t,t_{0},g_{0})$ remains in the compact set $\mathcal{N}_{b}(\pi(G_{c})))$. Hence, $\Phi_{\frac{1}{\omega}\hat{X}_{a}+\frac{1}{\omega^{2}}h}(\tau,\tau_{0},g_{0}))$ remains in $\mathcal{N}_{b}(\pi(G_{c})))$ since it is the same trajectory in $\tau$ scale. Consequently we have 
\EQ \bigcup_{\tau\in[\tau_{0},\infty)}\Phi_{\frac{1}{\omega} \hat{X}}(\tau,\tau_{0},g_{0})&\subset& \bigcup_{\tau\in[\tau_{0},\infty)}\Phi_{\frac{1}{\omega} Z}^{(1,0)}\circ \mathcal{N}_{b}(\pi(G_{c})))\nnum\\&=&\bigcup_{\tau\in[\tau_{0},\omega T]}\Phi_{\frac{1}{\omega} Z}^{(1,0)}\circ \mathcal{N}_{b}(\pi(G_{c}))),\nnum \EN
where the equality is due to the periodicity of $Z$. This proves that for all $g_{0}\in W_{\pi(G_{c})}$ the state trajectory $\Phi_{\frac{1}{\omega} \hat{X}}(\tau,\tau_{0},g_{0})$ is trapped in the compact set $\bigcup_{\tau\in[\tau_{0},\omega T]}\Phi_{\frac{1}{\omega} Z}^{(1,0)}\circ \mathcal{N}_{b}(\pi(G_{c})))$.

By the definition of the distance function given in (\ref{l}), we have  $d(\Phi_{\frac{1}{\omega}Z}(s,0,g),g)\leq \ell(\Phi_{\frac{1}{\omega} Z}(s,0,g),g),$ where $\ell(\Phi_{\frac{1}{\omega} Z}(s,0,g))$ is the length of the curve connecting $g$ to $\Phi_{\frac{1}{\omega} Z}(s,0,g)$ on $G^{m}/G_{c}$. Therefore,
\EQ \label{kirkoon}&&d(\Phi_{\frac{1}{\omega} Z}(1,0,g),g)\leq \nnum\\&&\ell(\Phi_{\frac{1}{\omega} Z}(1,0,g),g)= \frac{1}{\omega}\int^{1}_{0}||Z(\lambda,\Phi_{\frac{1}{\omega}Z}(s,0,g))||_{g}ds.\nnum\\\EN

Periodicity of $Z$ with respect to $\lambda$, boundedness of $\Phi_{\frac{1}{\omega}Z}(s,0,g),\hspace{.2cm}s\in[0,1]$ in the sense of compactness of $\bigcup_{\tau\in[\tau_{0},\omega T]}\Phi_{\frac{1}{\omega} Z}^{(1,0)}\circ \mathcal{N}_{b}(\pi(G_{c})))$ and smoothness of $Z$ with respect to $g$ together yield $d(\Phi_{\frac{1}{\omega} Z}(1,0,g),g)=O(\frac{1}{\omega}),\hspace{.2cm}g\in \mathcal{N}_{b}(\pi(G_{c}))$. Hence, we have 
\EQ d\left(\Phi_{\frac{1}{\omega} \hat{X}}(\tau,\tau_{0},g_{0}),\Phi_{\frac{1}{\omega} Z}^{(1,0)}\circ \Phi_{\frac{1}{\omega} \hat{X}}(\tau,\tau_{0},g_{0})\right)=O\left(\frac{1}{\omega}\right),\nnum\\\forall \tau\in[\tau_{0},\infty), g_{0}\in W_{\pi(G_{c})},\nnum\EN
where $g$ is replaced by $\Phi_{\frac{1}{\omega} \hat{X}}(\tau,\tau_{0},g_{0})\in G^{m}/G_{c}$.
Hence, by using (\ref{kirkir}), for any $g_{0}\in W_{\pi(G_{c})}$, there exists a time $T_{g_{0}}$, such that
\EQ &&d(\Phi_{\frac{1}{\omega} \hat{X}}(\tau,\tau_{0},g_{0}), \Phi_{\frac{1}{\omega} \hat{X}_{a}}(\tau,\tau_{0},g_{0}))\leq\nnum\\&& d(\Phi_{\frac{1}{\omega} \hat{X}}(\tau,\tau_{0},g_{0}),\Phi_{\frac{1}{\omega} Z}^{(1,0)}\circ \Phi_{\frac{1}{\omega} \hat{X}}(\tau,\tau_{0},g_{0}))+\nnum\\&& d(\Phi_{\frac{1}{\omega} Z}^{(1,0)}\circ \Phi_{\frac{1}{\omega} \hat{X}}(\tau,\tau_{0},g_{0}),\Phi_{\frac{1}{\omega} \hat{X}_{a}}(\tau,\tau_{0},g_{0}))\leq \nnum\\&&\nnum\\&&d(\Phi_{\frac{1}{\omega} \hat{X}}(\tau,\tau_{0},g_{0}),\Phi_{\frac{1}{\omega} Z}^{(1,0)}\circ \Phi_{\frac{1}{\omega} \hat{X}}(\tau,\tau_{0},g_{0}))+\nnum\\&&d(\Phi_{\frac{1}{\omega} Z}^{(1,0)}\circ \Phi_{\frac{1}{\omega} \hat{X}}(\tau,\tau_{0},g_{0}),\pi(G_{c}))+\nnum\\&&d(\pi(G_{c}),\Phi_{\frac{1}{\omega} \hat{X}_{a}}(\tau,\tau_{0},g_{0}))\leq O\left(\frac{1}{\omega}\right)+\nnum\\&& \rho(||T_{g^{m}}\pi\left(\sum^{m}_{j=1}\sum^{n}_{i=1}O((\max_{j\in1,\cdots,m, i\in1,\cdots,n}a^{j}_{i})^{4}) g_{j}\frac{\partial}{\partial g_{i}}\right)+\nnum\\&&\frac{1}{\omega}h(g,\zeta,\frac{\tau}{\omega})||_{g})+\nnum\\&&\hat{\rho}\left(||T_{g^{m}}\pi\left(\sum^{m}_{j=1}\sum^{n}_{i=1}O((\max_{j\in1,\cdots,m, i\in1,\cdots,n}a^{j}_{i})^{4}) g_{j}\frac{\partial}{\partial g_{i}}\right)||_{g}\right),\hspace{.2cm}\nnum\\&&\forall \tau\in [\omega T_{g_{0}},\infty), g_{0}\in W_{\pi(G_{c})},\nnum\EN
where $\hat{\rho}$ is derived by Lemma \ref{l5}. 
Note that $\Phi_{\frac{1}{\omega} Z}^{(1,0)}\circ \Phi_{\frac{1}{\omega} \hat{X}}(\tau,\tau_{0},g_{0})=\Phi_{\hat{X}+\frac{1}{\omega}h}(t,t_{0},g_{0})$, for $\tau=\omega t$ and $\tau_{0}=\omega t_{0}$.
Finally we have 
\EQ &&d(\Phi_{\frac{1}{\omega} \hat{X}}(\tau,\tau_{0},g_{0}),\pi(G_{c}))\leq \nnum\\&& d(\Phi_{\frac{1}{\omega} \hat{X}}(\tau,\tau_{0},g_{0}), \Phi_{\frac{1}{\omega} \hat{X}_{a}}(\tau,\tau_{0},g_{0}))+ \nnum\\&&d(\pi(G_{c}), \Phi_{\frac{1}{\omega} \hat{X}_{a}}(\tau,\tau_{0},g_{0}))\leq O\left(\frac{1}{\omega}\right)+\nnum\\&&\nnum\\&& \rho(||T_{g^{m}}\pi\left(\sum^{m}_{j=1}\sum^{n}_{i=1}O((\max_{j\in1,\cdots,m, i\in1,\cdots,n}a^{j}_{i})^{4}) g_{j}\frac{\partial}{\partial g_{i}}\right)+\nnum\\&&\frac{1}{\omega}h(g,\zeta,\frac{\tau}{\omega})||_{g})+\nnum\\&&2\hat{\rho}||T_{g^{m}}\pi\left(\sum^{m}_{j=1}\sum^{n}_{i=1}O((\max_{j\in1,\cdots,m, i\in1,\cdots,n}a^{j}_{i})^{4}) g_{j}\frac{\partial}{\partial g_{i}}\right)||_{g},\hspace{.2cm}\nnum\\&&\forall \tau\in [\omega T_{g_{0}},\infty),g_{0}\in W_{\pi(G_{c})}.\nnum\EN
 Following the proof of Lemma \ref{l5} we can show
 \EQ &&d\left(\Phi_{ X}(t,t_{0},g^{m}_{0}),G_{c}\right)\leq O\left(\frac{1}{\omega}\right)+\nnum\\&&\nnum\\&& \rho(||T_{g^{m}}\pi\left(\sum^{m}_{j=1}\sum^{n}_{i=1}O((\max_{j\in1,\cdots,m, i\in1,\cdots,n}a^{j}_{i})^{4}) g_{j}\frac{\partial}{\partial g_{i}}\right)+\nnum\\&&\frac{1}{\omega}h(g,\zeta,t)||_{g})+\nnum\\&&2\hat{\rho}||T_{g^{m}}\pi\left(\sum^{m}_{j=1}\sum^{n}_{i=1}O((\max_{j\in1,\cdots,m, i\in1,\cdots,n}a^{j}_{i})^{4}) g_{j}\frac{\partial}{\partial g_{i}}\right)||_{g},\hspace{.2cm}\nnum\\&&\forall t\in [T_{g_{0}},\infty),\pi(g^{m}_{0})\in W_{\pi(G_{c})},\nnum\EN
 which gives the closeness of solutions on $G^{m}$ and completes the proof for $\hat{U}_{\pi(G_{c})}= W_{\pi(G_{c})}$.
  \end{proof}
  \section{Example on $SO(3)$}
  In this section we present a simple example for synchronization of three agents evolving on $SO(3)$ as their ambient state manifold. For this problem the synchronization cost is given by
  \EQ\label{cc} &&J:SO(3)\times SO(3)\times SO(3)\rightarrow \mathds{R},\nnum\\&& J(g_{1},g_{2},g_{3})=\frac{1}{2}tr\left((g_{1}-g_{2})^{T}\cdot (g_{1}-g_{2})\right)+\nnum\\&&\frac{1}{2}tr\left((g_{1}-g_{3})^{T}\cdot (g_{1}-g_{3})\right)+\nnum\\&&\frac{1}{2}tr\left((g_{2}-g_{3})^{T}\cdot (g_{2}-g_{3})\right),\hspace{.2cm}g_{1},g_{2},g_{3}\in SO(3).\nnum\\\EN
  The invariant synchronization set is given by $G_{c}=(g_{c},g_{c},g_{c})\in SO^{3}(3)$, where one can verify that $J(g_{c}\star g_{1},g_{c}\star g_{2},g_{c}\star g_{3})=J(g_{1},g_{2},g_{3})$ for all $g_{c}\in SO(3)$. Hence, $J$ is $G_{c}$ invariant. 
  The Lie algebra $so(3)$ is spanned by $\frac{\partial}{\partial g_{1}}=\left( \begin{array}{ll} 0 \quad\hspace{.3cm} 1 \quad 0\\-1\quad 0\quad 0\\0\quad\hspace{.3cm} 0\quad 0
        \end{array}\right), \frac{\partial}{\partial g_{2}}=\left( \begin{array}{ll} 0 \quad\hspace{.3cm} 0 \quad 0\\0\quad\hspace{.3cm} 0\quad 1\\0\quad -1\quad 0
        \end{array}\right) \mbox{and}\hspace{.2cm} \frac{\partial}{\partial g_{3}}=\left( \begin{array}{ll} 0 \quad\hspace{.3cm} 0 \quad 1\\0\quad\hspace{.3cm} 0\quad 0\\-1\quad 0\quad 0
        \end{array}\right)$. For this example the dither vector $X(e)$ at the Lie algebra $so(3)$ is given by
				\EQ X(e)&&=\sum^{3}_{i=1}a_{i}\sin(\omega_{i}t)\frac{\partial}{\partial g_{i}}\nnum\\&&=\left( \begin{array}{ll} \hspace{1cm}0 \quad\hspace{.3cm} a_{1}\sin(\omega_{1}t) \quad a_{3}\sin(\omega_{3}t)\\-a_{1}\sin(\omega_{1}t)\quad \hspace{.5cm}0\quad  \hspace{.5cm}a_{2}\sin(\omega_{2}t)\\-a_{3}\sin(\omega_{3}t)\quad\hspace{-.3cm} -a_{2}\sin(\omega_{2}t)\quad 0
        \end{array}\right),\nnum\\\EN
				hence, the dither vector field is given by
				\EQ X(g)=g\cdot \left( \begin{array}{ll} \hspace{1cm}0 \quad\hspace{.3cm} a_{1}\sin(\omega_{1}t) \quad a_{3}\sin(\omega_{3}t)\\-a_{1}\sin(\omega_{1}t)\quad \hspace{.5cm}0\quad  \hspace{.5cm}a_{2}\sin(\omega_{2}t)\\-a_{3}\sin(\omega_{3}t)\quad\hspace{-.3cm} -a_{2}\sin(\omega_{2}t)\quad 0
        \end{array}\right),\nnum\EN
				where $g\in SO(3)$.
				
				\begin{figure}
\begin{center}
\vspace{0cm}
\hspace*{0cm}
\includegraphics[scale=.3]{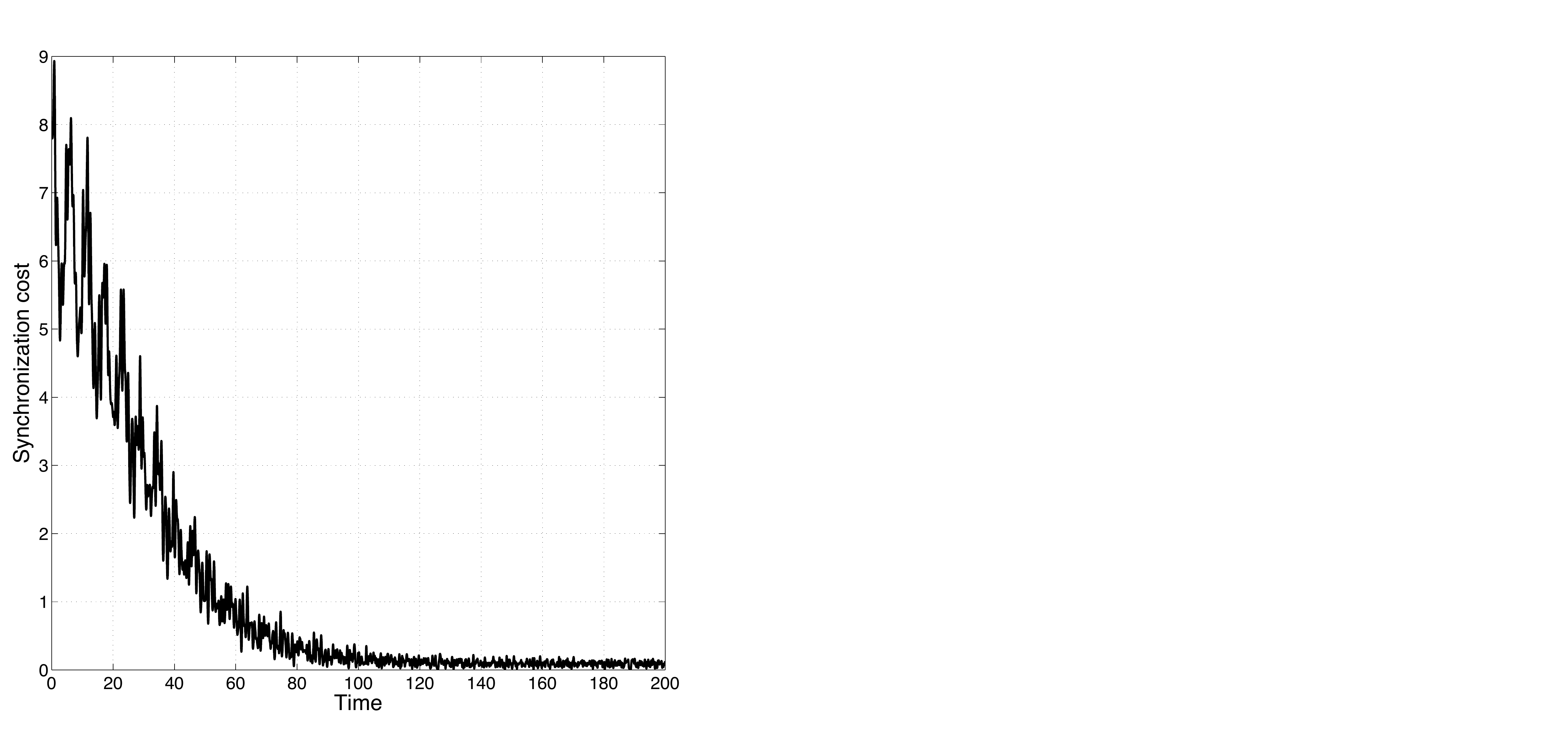}
 \caption{  Convergence of the synchronization cost function }
      \label{e1}
      \end{center}
   \end{figure} 
\begin{figure}
\begin{center}
\vspace{0cm}
\hspace*{0cm}
\includegraphics[scale=.3]{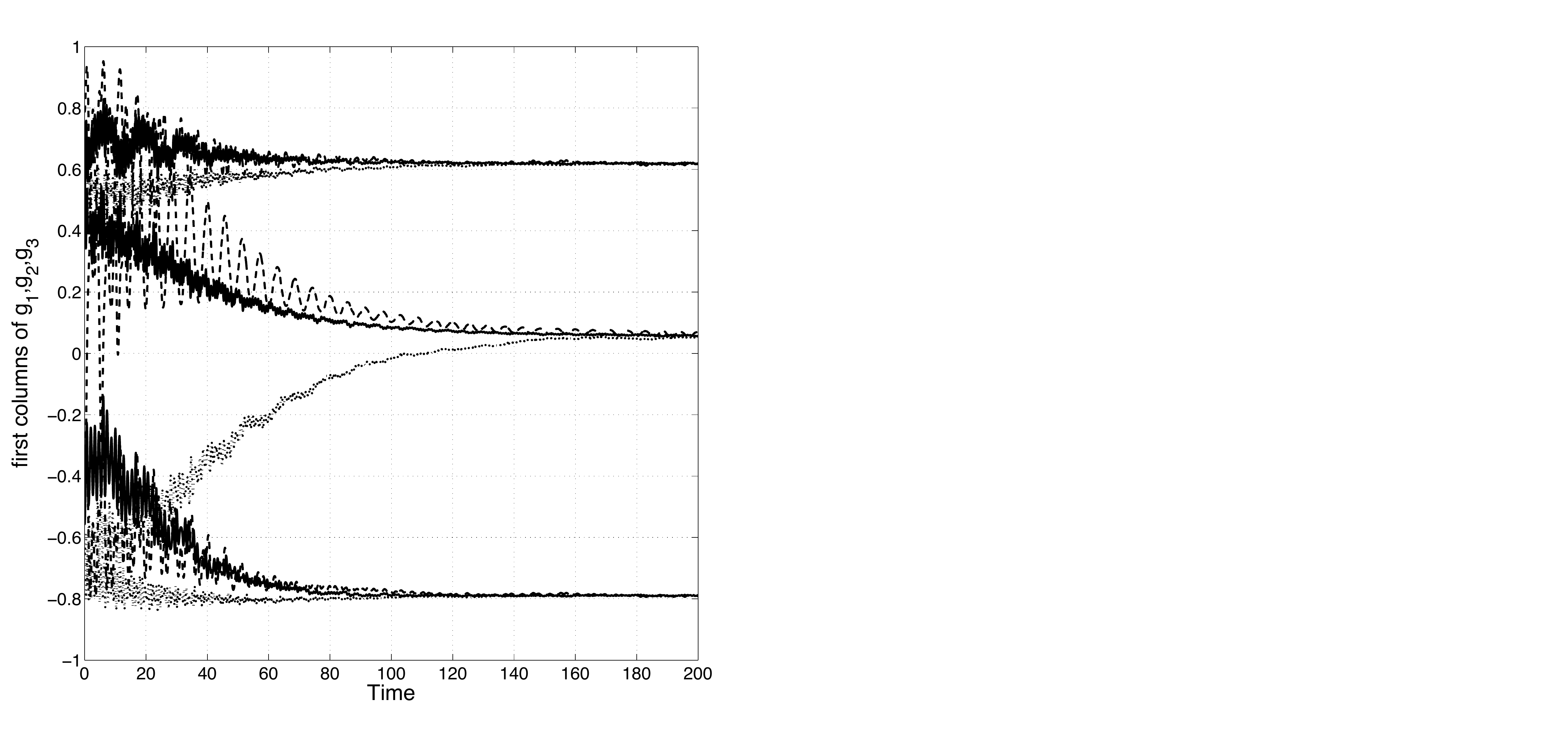}
 \caption{  Convergence of the first column of $g_{1},g_{2},g_{3}$}
      \label{e2}
      \end{center}
   \end{figure} 

\begin{figure}
\begin{center}
\vspace{0cm}
\hspace*{0cm}
\includegraphics[scale=.3]{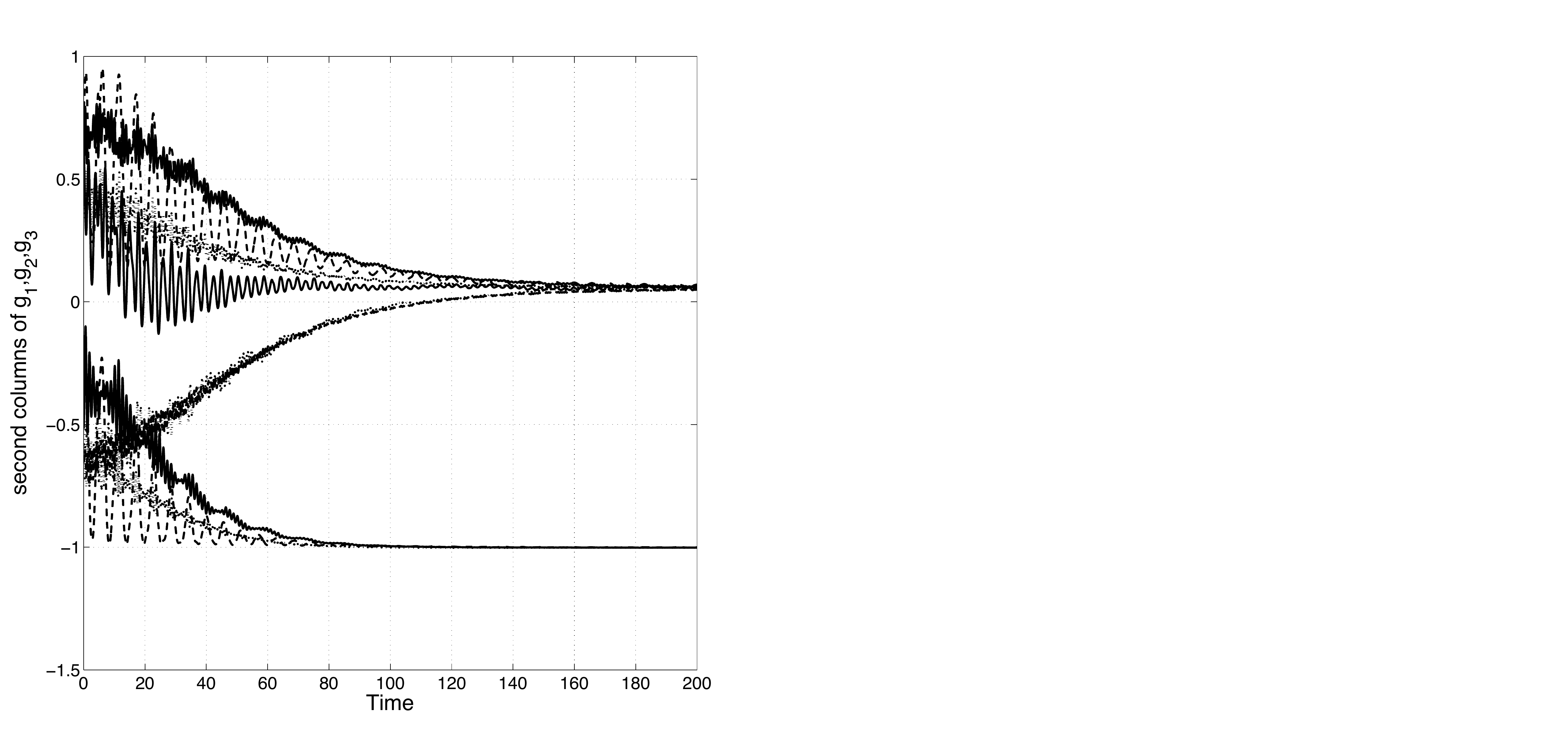}
 \caption{ Convergence of the second column of $g_{1},g_{2},g_{3}$}
      \label{e3}
      \end{center}
   \end{figure} 
   \begin{figure}
\begin{center}
\vspace{0cm}
\hspace*{0cm}
\includegraphics[scale=.3]{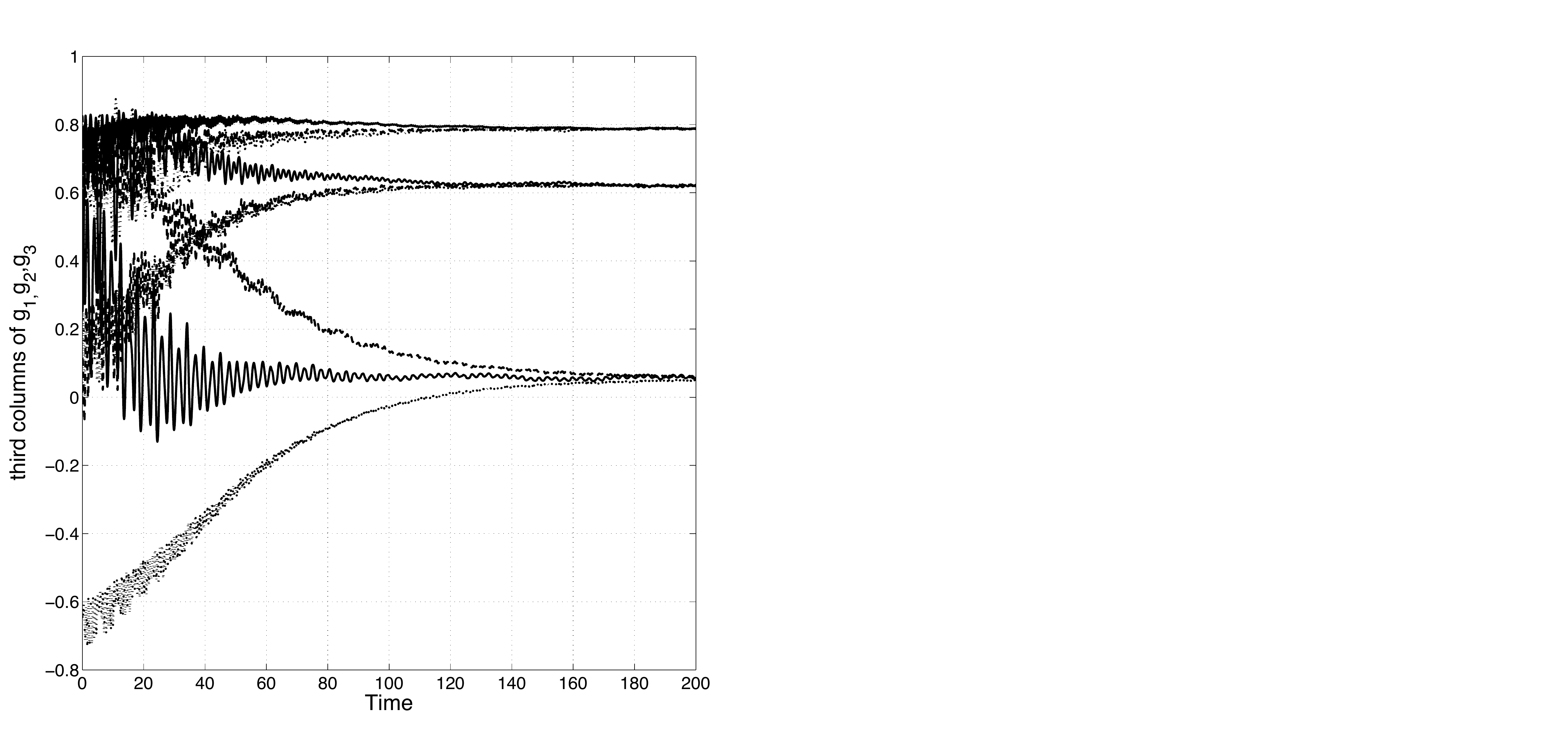}
 \caption{ Convergence of the third column of $g_{1},g_{2},g_{3}$}
      \label{e4}
      \end{center}
   \end{figure} 
The extremum seeking for synchronization of the agents in this example is given by
\EQ\label{po} &&\hspace{-1cm}\dot{g}_{j}=-\sum^{3}_{i=1}a^{j}_{i}\sin(\omega_{i}t)\times J\Big(g_{1}\cdot\exp(\sum^{3}_{i=1}a^{1}_{i}\sin(\omega^{1}_{i}t)\frac{\partial}{\partial g_{i}}),\nnum\\&&\hspace{-1cm}g_{2}\cdot\exp(\sum^{3}_{i=1}a^{2}_{i}\sin(\omega^{2}_{i}t)\frac{\partial}{\partial g_{i}}),\nnum\\&&\hspace{-1cm}g_{3}\cdot\exp(\sum^{3}_{i=1}a^{3}_{i}\sin(\omega^{3}_{i}t)\frac{\partial}{\partial g_{i}})\Big)g_{j}\frac{\partial}{\partial g_{i}},j=1,2,3.\EN

The initial configuration of agents are given by $g_{1}=\left( \begin{array}{ll} -0.3766\quad   -0.8917\quad    0.2509\\0.7877\quad   -0.1658\quad    0.5934\\-0.4875\quad    0.4211\quad    0.7648\end{array}\right), g_{2}=\left( \begin{array}{ll} -0.5569\quad    0.8229\quad    0.1122\\ -0.6528\quad   -0.5173\quad    0.5534\\0.5134\quad    0.2350\quad    0.8253\end{array}\right), g_{3}=\left( \begin{array}{ll} -0.6536\quad   -0.7568\quad         0\\  0.5788\quad   -0.4999\quad   -0.6442\\0.4875   -0.4211\quad    0.7648\end{array}\right)$. Figure \ref{e1} shows the convergence of the synchronization algorithm in terms of minimizing (\ref{cc}) for a proper set of frequencies  $\omega^{j}_{i},i,j=1,2,3$. Figures \ref{e2}-\ref{e4} show the synchronization of $g_{1},g_{2},g_{3}$ on $SO^{3}(3)$.

\section{Example on $SE(3)$}
\label{s5}
	In this section we give another conceptual example  for an orientation control on $SE(3)$. 

As is known, $SE(3)$ is the space of rotation and translation which is used for robotic modeling. We have
\EQ SE(3)= \big \{\left( \begin{array}{ll} g^{SO(3)} \hspace{.3cm} g^{\mathds{R}} \\ 0_{1\times 3}\hspace{.6cm}  1
        \end{array}\right)\in \mathds{R}^{4\times 4}|\nnum\\ \quad g^{SO(3)}\in SO(3), g^{\mathds{R}}\in \mathds{R}^{3\times 1}\big\},\nnum\EN
				where $g^{SO(3)}$ models the rotation and $g^{\mathds{R}}$ models the translation in $\mathds{R}^{3}$.
The Lie algebra of $SE(3)$ which is denoted by $se(3)$ is given by
 \EQ se(3)=\big\{  \left( \begin{array}{ll} S \hspace{.7cm} v \\ 0_{1\times 3}\hspace{.2cm}  0
        \end{array}\right)\in \mathds{R}^{4\times 4}|\quad S\in so(3), v\in \mathds{R}^{3}\big\},\nnum\EN

Let us consider the  synchronization cost function for three agents as $J:SE^{3}(3)\rightarrow \mathds{R}$, which is given by
\EQ \label{cost2}J(g)&=&\frac{1}{2}tr\left((g^{SO(3)}_{1}-g^{SO(3)}_{2})^{T}\cdot (g^{SO(3)}_{1}-g^{SO(3)}_{2})\right)+\nnum\\&&\frac{1}{2}tr\left((g^{SO(3)}_{1}-g^{SO(3)}_{3})^{T}\cdot (g^{SO(3)}_{1}-g^{SO(3)}_{3})\right)+\nnum\\&&\frac{1}{2}tr\left((g^{SO(3)}_{2}-g^{SO(3)}_{3})^{T}\cdot (g^{SO(3)}_{2}-g^{SO(3)}_{3})\right)+\nnum\\&&\frac{1}{2}||g^{\mathds{R}}_{1}-g^{\mathds{R}}_{2}||_{\mathds{R}^{3}}^{2}+\frac{1}{2}||g^{\mathds{R}}_{1}-g^{\mathds{R}}_{3}||_{\mathds{R}^{3}}^{2}+\nnum\\&&\frac{1}{2}||g^{\mathds{R}}_{2}-g^{\mathds{R}}_{3}||_{\mathds{R}^{3}}^{2}.\EN

The synchronization set for this problem is given by $G_{c}=(g_{c},g_{c},g_{c})\in SE^{3}(3)$. One can verify that $G_{c}$ is invariant for the 
cost function(\ref{cost2}). Since the group operation on $SE(3)$ is given by matrix multiplication then we have 
\EQ g_{c}\cdot g_{j}&=&\left( \begin{array}{ll} g^{SO(3)}_{c} \hspace{.3cm} g^{\mathds{R}}_{c} \\ 0_{1\times 3}\hspace{.6cm}  1
        \end{array}\right)\cdot \left( \begin{array}{ll} g^{SO(3)}_{j} \hspace{.3cm} g^{\mathds{R}}_{j} \\ 0_{1\times 3}\hspace{.6cm}  1
        \end{array}\right)\nnum\\&=&\left( \begin{array}{ll} g^{SO(3)}_{c}.g^{SO(3)}_{j} \hspace{.3cm} g^{SO(3)}_{c}\cdot g^{\mathds{R}}_{j}+g^{\mathds{R}}_{c}\\ 0_{1\times 3}\hspace{2.6cm}  1
        \end{array}\right),\nnum\\&&j=1,2,3.\nnum\EN
        It is immediate that the rotation terms in (\ref{cost2}) are invariant with respect to $SO(3)$. Also the displacement terms are given by $\frac{1}{2}||g^{SO(3)}_{c}\cdot g^{\mathds{R}}_{i}+g^{\mathds{R}}_{c}-g^{SO(3)}_{c}\cdot g^{\mathds{R}}_{j}-g^{\mathds{R}}_{c}||_{\mathds{R}^{3}}^{2}=\frac{1}{2}||g^{\mathds{R}}_{i}-g^{\mathds{R}}_{j}||_{\mathds{R}^{3}}^{2}$. Hence, (\ref{cost2}) is $G_{c}$ invariant.

The Lie algebra $se(3)$ is spanned by $\frac{\partial}{\partial g_{1}}=\left( \begin{array}{ll} 0 \quad 1 \quad 0 \quad 0\\\hspace{-.25cm}-1\quad 0\quad 0\quad 0\\0\quad 0\quad 0\quad 0\\0\quad 0\quad 0\quad0
        \end{array}\right), \frac{\partial}{\partial g_{2}}=\left( \begin{array}{ll} 0 \quad 0 \quad 1 \quad 0\\0\quad 0\quad 0\quad 0\\\hspace{-.25cm}-1\quad 0\quad 0\quad 0\\0\quad 0\quad 0\quad0
        \end{array}\right),\frac{\partial}{\partial g_{3}}=\left( \begin{array}{ll} 0 \quad 0 \quad 0 \quad 0\\0\quad 0\quad 1\quad 0\\0\quad \hspace{-.3cm}-1\quad 0\quad 0\\0\quad 0\quad 0\quad0
        \end{array}\right),\frac{\partial}{\partial g_{4}}=\left( \begin{array}{ll} 0 \quad 0 \quad 0 \quad 1\\0\quad 0\quad 0\quad 0\\0\quad 0\quad 0\quad 0\\0\quad 0\quad 0\quad0
        \end{array}\right), \frac{\partial}{\partial g_{5}}=\left( \begin{array}{ll} 0 \quad 0 \quad 0 \quad 0\\0\quad 0\quad 0\quad 1\\0\quad 0\quad 0\quad 0\\0\quad 0\quad 0\quad0
        \end{array}\right) \mbox{and}\hspace{.2cm} \frac{\partial}{\partial g_{6}}=\left( \begin{array}{ll} 0 \quad 0 \quad 0 \quad 0\\0\quad 0\quad 0\quad 0\\0\quad 0\quad 0\quad 1\\0\quad 0\quad 0\quad0
        \end{array}\right)$. For this example the dither vector $X(e)$ at the Lie algebra $se(3)$ is given by
				\EQ X(e)&&=\sum^{6}_{i=1}a_{i}\sin(\omega_{i}t)\frac{\partial}{\partial g_{i}}\nnum\\&&\hspace{-2cm}=\left( \begin{array}{ll} \hspace{1cm}0 \quad\hspace{.3cm} a_{1}\sin(\omega_{1}t) \quad a_{3}\sin(\omega_{3}t)\quad a_{4}\sin(\omega_{4}t)\\-a_{1}\sin(\omega_{1}t)\quad \hspace{.5cm}0\quad  \hspace{.5cm}a_{2}\sin(\omega_{2}t)\quad a_{5}\sin(\omega_{5}t)\\-a_{3}\sin(\omega_{3}t)\quad\hspace{-.3cm} -a_{2}\sin(\omega_{2}t)\quad 0\quad\hspace{.9cm} a_{6}\sin(\omega_{6}t)\\\hspace{1cm}0\quad\hspace{1.5cm} 0\quad\hspace{1cm} 0\quad \hspace{1.5cm}0
        \end{array}\right),\nnum\EN
				hence, the dither vector field is given by $X(g)=g\cdot X(e),$
				where $g\in SE(3)$.

Similar to the example on $SO(3)$, the extremum seeking vector field on $SE(3)$ is given by the following vector field
\EQ\label{ll} &&\hspace{-1cm}-\sum^{6}_{i=1}a_{i}\sin(\omega_{i}t)J(\cdot,g_{j}\exp\sum^{6}_{i=1}a_{i}\sin(\omega_{i}t)\frac{\partial}{\partial g_{i}},\cdot)g_{j}\frac{\partial}{\partial g_{i}},\nnum\\\EN
where $\exp$ is the exponential operator defined on $SE(3)$. In this case, the $\exp$ operator is not the same as the $\exp$ operator on $SO(3)$. For a tangent vector $\left( \begin{array}{ll} S \hspace{.7cm} v \\ 0_{1\times 3}\hspace{.2cm}  0\end{array}\right)\in se(3)$, where $S=\left( \begin{array}{ll}\hspace{.25cm} 0 \quad a \quad b \\-a\quad 0\quad c\\-b\quad \hspace{-.25cm}-c\quad 0\end{array}\right)$, we have
$ \exp(\left( \begin{array}{ll} S \hspace{.7cm} v \\ 0_{1\times 3}\hspace{.2cm}  0\end{array}\right))=\left( \begin{array}{ll} \exp(S) \hspace{.1cm} Av \\ 0_{1\times 3}\hspace{.6cm}  1\end{array}\right),$
where $A=I_{3\times 3}+\frac{(1-\cos(\theta))}{\theta^{2}}S+\frac{(\theta-\sin(\theta))}{\theta^{3}}S^{2}$, and $\theta=\sqrt{a^{2}+b^{2}+c^{2}}$. In the case that $\theta=0$, we have $\exp(\left( \begin{array}{ll} S \hspace{.7cm} v \\ 0_{1\times 3}\hspace{.2cm}  0\end{array}\right))=\left( \begin{array}{ll} \exp(S) \hspace{.2cm} v \\ 0_{1\times 3}\hspace{.6cm}  1\end{array}\right)$.

	The extremum seeking for synchronization of the agents in this example is given by
\EQ\label{po} &&\hspace{-1cm}\dot{g}_{j}=-\sum^{6}_{i=1}a^{j}_{i}\sin(\omega_{i}t)\times J\Big(g_{1}\cdot\exp(\sum^{6}_{i=1}a^{1}_{i}\sin(\omega^{1}_{i}t)\frac{\partial}{\partial g_{i}}),\nnum\\&&\hspace{-1cm}g_{2}\cdot\exp(\sum^{6}_{i=1}a^{2}_{i}\sin(\omega^{2}_{i}t)\frac{\partial}{\partial g_{i}}),\nnum\\&&\hspace{-1cm}g_{3}\cdot\exp(\sum^{3}_{i=1}a^{3}_{i}\sin(\omega^{3}_{i}t)\frac{\partial}{\partial g_{i}})\Big)g_{j}\frac{\partial}{\partial g_{i}},j=1,2,3.\EN
The initial configuration of agents are given by $g_{1}=\left( \begin{array}{ll} -0.3766\quad   -0.8917\quad    0.2509\quad    5\\0.7877\quad   -0.1658\quad    0.5934\quad\quad    1\\-0.4875\quad    0.4211\quad    0.7648\quad\quad    1\\ 0\qquad\qquad         0\qquad\qquad          0\qquad\qquad    1\end{array}\right), g_{2}=\left( \begin{array}{ll} -0.5165\quad   -0.8489\quad    0.1122\quad    4\\  0.7420\quad   -0.3784\quad    0.5534\quad\quad    2\\-0.4273\quad    0.3691\quad    0.8253\quad \quad   1\\0\qquad\qquad         0\qquad\qquad          0\qquad\qquad    1\end{array}\right), g_{3}=\left( \begin{array}{ll} -0.2961\quad   -0.9038\quad    0.3089\quad    5\\  0.8213\quad   -0.0759\quad    0.5654\quad \quad   2\\-0.4875\quad    0.4211\quad    0.7648\quad \quad   0\\0\qquad\qquad         0\qquad\qquad          0\qquad\qquad    1\end{array}\right)$. Figure \ref{e1} shows the convergence of the synchronization algorithm in terms of minimizing (\ref{cc}) for a proper set of frequencies  $\omega^{j}_{i},j=1,2,3, i=1,\cdots,6$. Figures \ref{e5}-\ref{e9} show the synchronization of $g_{1},g_{2},g_{3}$ on $SE^{3}(3)$.

\begin{figure}
\begin{center}
\vspace{0cm}
\hspace*{0cm}
\includegraphics[scale=.3]{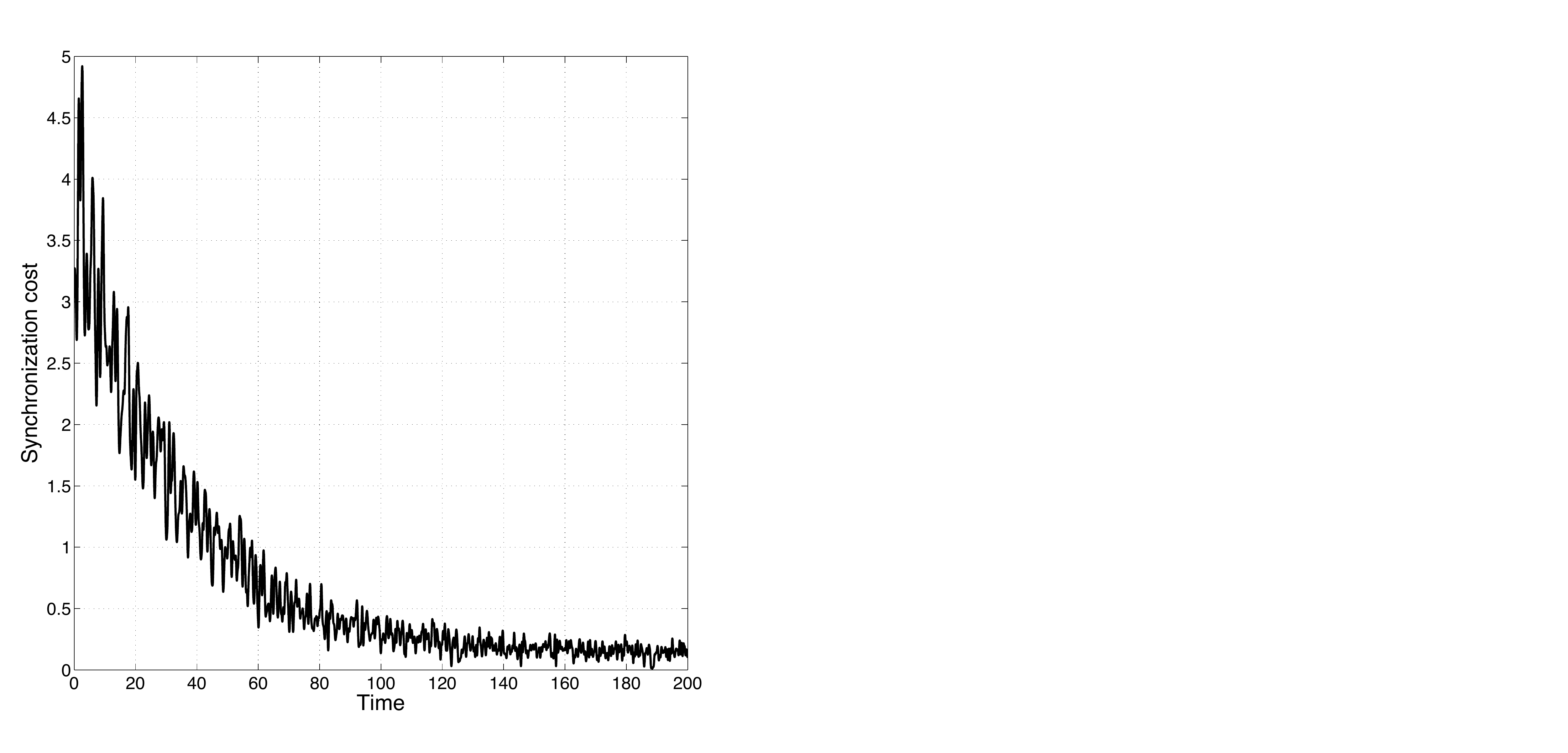}
 \caption{ Convergence of the synchronization cost on $SE(3)$}
      \label{e5}
      \end{center}
   \end{figure} 
	\begin{figure}
\begin{center}
\vspace{0cm}
\hspace*{0cm}
\includegraphics[scale=.3]{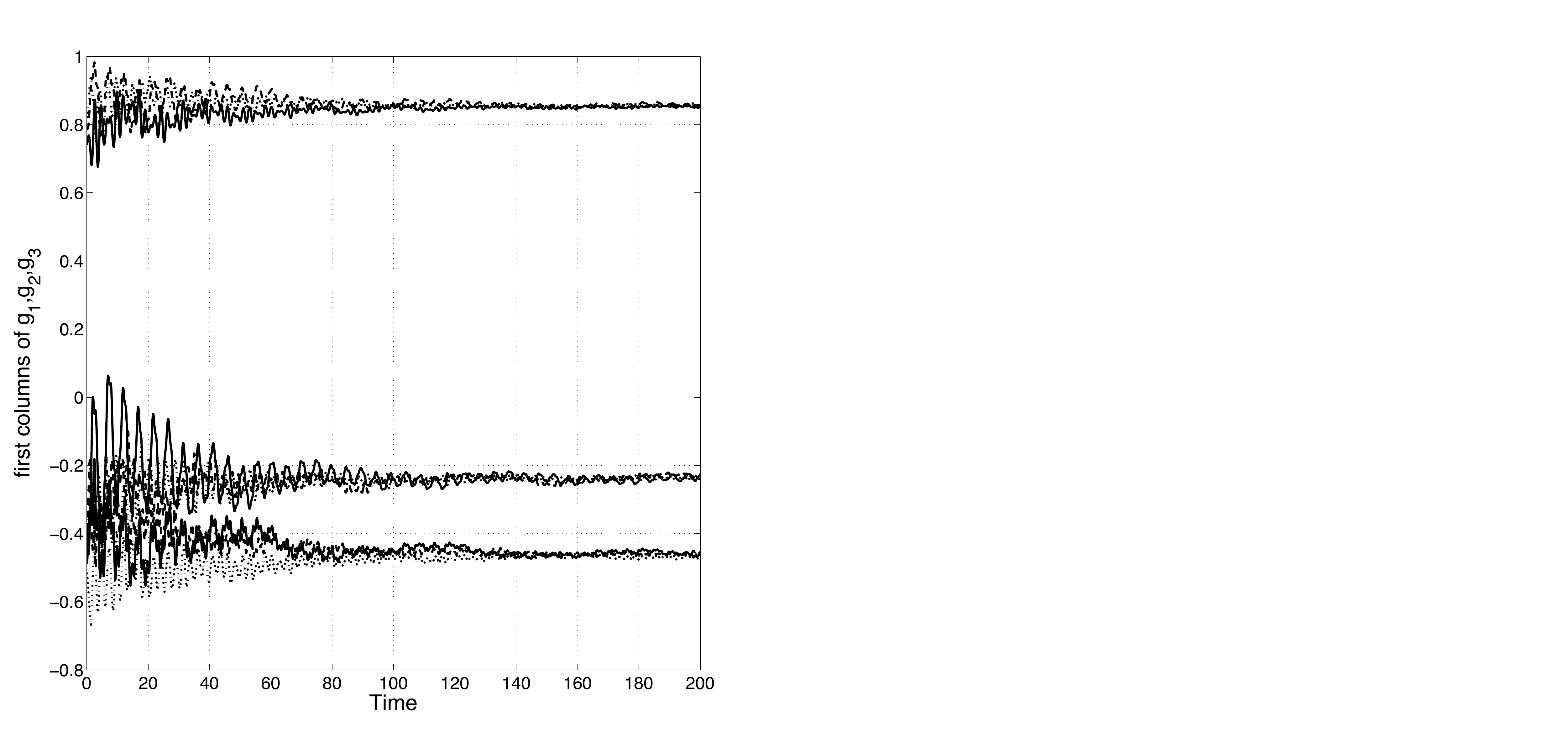}
 \caption{ Convergence of the first columns of  $g_{1},g_{2},g_{3}$}
      \label{e6}
      \end{center}
   \end{figure} 
	\begin{figure}
\begin{center}
\vspace{0cm}
\hspace*{0cm}
\includegraphics[scale=.3]{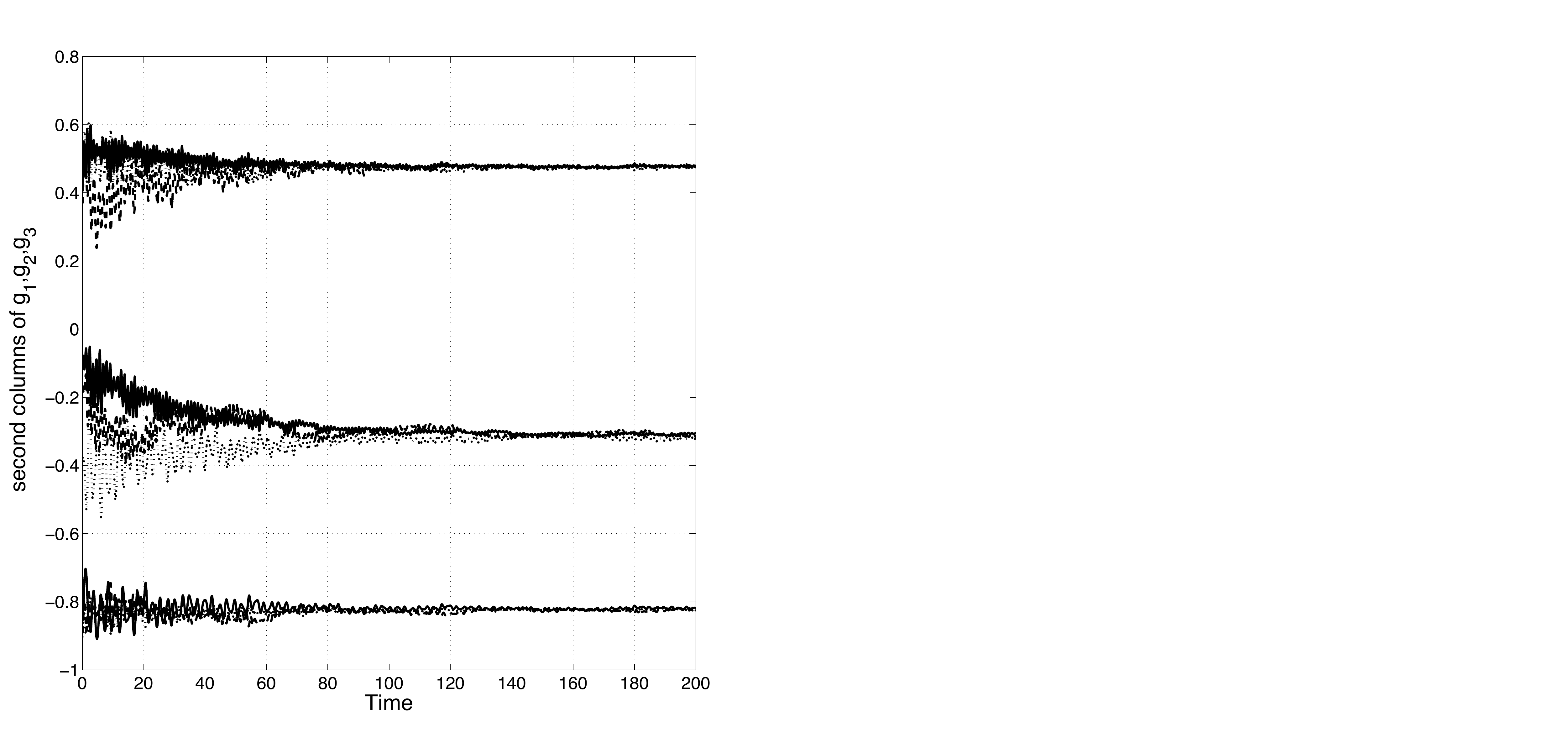}
 \caption{ Convergence of the second columns of  $g_{1},g_{2},g_{3}$}
      \label{e7}
      \end{center}
   \end{figure} 
	\begin{figure}
\begin{center}
\vspace{0cm}
\hspace*{0cm}
\includegraphics[scale=.3]{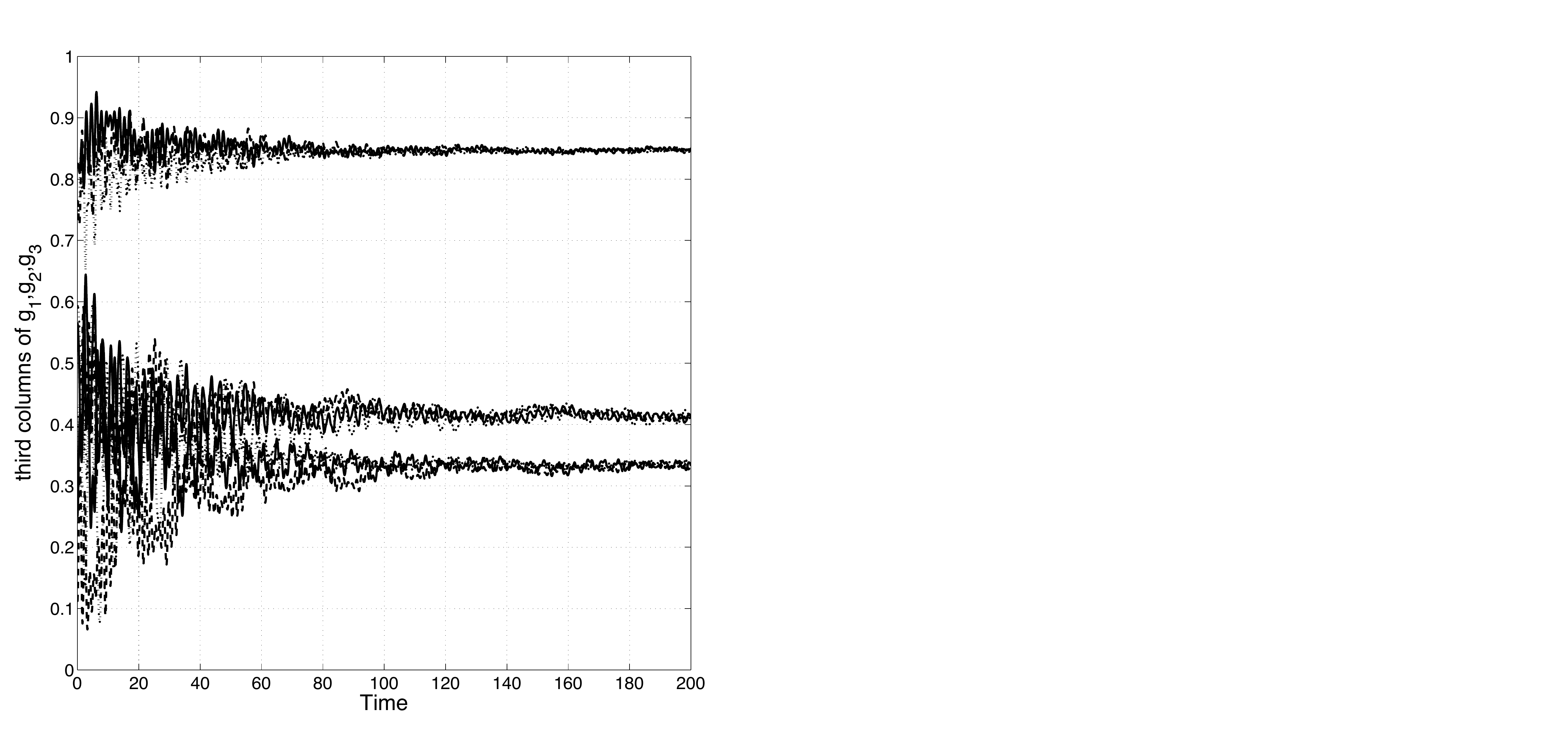}
 \caption{ Convergence of the third columns of  $g_{1},g_{2},g_{3}$}
      \label{e8}
      \end{center}
   \end{figure} 
	\begin{figure}
\begin{center}
\vspace{0cm}
\hspace*{0cm}
\includegraphics[scale=.3]{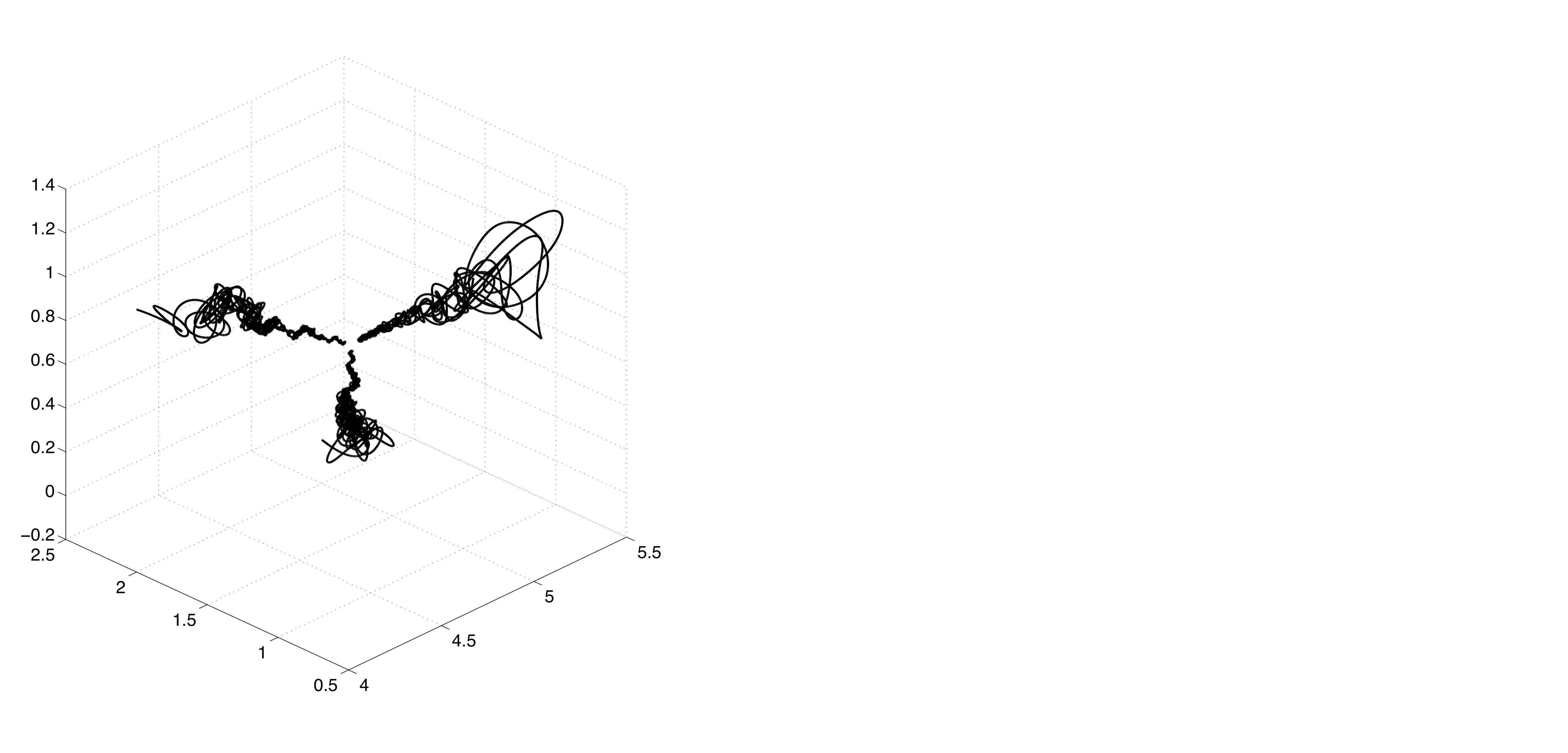}
 \caption{ Synchronization  of $g_{1},g_{2},g_{3}$ in $\mathds{R}^{3}$}
      \label{e9}
      \end{center}
   \end{figure} 
\bibliographystyle{ieeetr}
\bibliography{HSCC1}\end{document}